\documentclass[english, 11pt, letterpaper]{article}
\usepackage{amsmath, amsthm, amssymb}
\usepackage{lmodern}
\usepackage[T1]{fontenc}
\usepackage[utf8]{inputenc}
\usepackage{enumitem}
\usepackage[numbers]{natbib}
\usepackage{subcaption}

\makeatletter
\newtheorem*{theorem*}{Theorem}
\newtheorem{theorem}{Theorem}
\theoremstyle{plain}
\newtheorem{lem}{\protect\lemmaname}
\theoremstyle{plain}
\newtheorem{prop}{\protect\propositionname}
\newtheorem{definition}{Definition}

\RequirePackage[dvipsnames]{xcolor}
\input{macros.tex}

\usepackage[
        unicode=true,
        pdfusetitle,
        bookmarks=true,
        bookmarksnumbered=false,
        bookmarksopen=false,
        breaklinks=true,
        pdfborder={0 0 0},
        pdfborderstyle={},
        backref=false,
        colorlinks=true]
        {hyperref}
\hypersetup{urlcolor=red!70!blue, citecolor=blue}

\usepackage{geometry}
\makeatother

\usepackage{babel}
\providecommand{\lemmaname}{Lemma}
\providecommand{\propositionname}{Proposition}

\begin{document}
\title{Low-Rank Matrix Estimation From Rank-One Projections by Unlifted Convex Optimization}

\author{
    Sohail Bahmani\footnote{School of Electrical and Computer Engineering, Georgia Institute of Technology, Atlanta, GA 30332}\\ \href{mailto:sohail.bahmani@ece.gatech.edu}{\texttt{sohail.bahmani@ece.gatech.edu}}
    \and
    Kiryung Lee\footnote{Department of Electrical and Computer Engineering, The Ohio State University, Columbus, OH 43210}\\
    \href{mailto:lee.8763@osu.edu}{\texttt{lee.8763@osu.edu}}
}

\maketitle

\begin{abstract}
We study an estimator with a convex formulation for recovery of low-rank matrices from rank-one projections. Using initial estimates of the factors of the target $d_1\times d_2$ matrix of rank-$r$, the estimator admits a practical subgradient method operating in a space of dimension $r(d_1+d_2)$. This property makes the estimator significantly more scalable than the convex estimators based on lifting and semidefinite programming. Furthermore, we present a streamlined analysis for exact recovery under the real Gaussian measurement model, as well as the partially derandomized measurement model by using the spherical $t$-design. We show that under both models the estimator succeeds, with high probability, if the number of measurements exceeds $r^2 (d_1+d_2)$ up to some logarithmic factors. This sample complexity improves on the existing results for nonconvex iterative algorithms. 
\end{abstract}

\section{Introduction}

We consider the problem of estimating a matrix $\mb M_{0}\in\mbb C^{d_{1}\times d_{2}}$ of known rank $r\!\ll\!\min\{d_1,d_2\}$ from rank-one ``sketches'' of the form
\begin{align}
m_i & = \mb a_i^* \mb M_{0}\mb b_i\,, & i=1,\dotsc,n\,, \label{eq:measurements}
\end{align}
for random vectors $\mb a_i\in \mbb C^{d_1}$ and $\mb b_i\in\mbb C^{d_2}$ drawn from certain distributions. More specifically, given the observations $\left\{ \left(\mb a_i,\mb b_i,m_i\right)\right\}_{i=1}^{n}$, the goal is to estimate factors $\mb X_0 \in \mbb C^{d_1\times r}$ and $\mb Y_0 \in \mbb C^{d_2\times r}$ of $\mb M_0$ (i.e., $\mb M_0=\mb X_0 \mb Y_0^*$). 

Depending on the distribution of $(\mb a_i,\mb b_i)_i$, the observation model \eqref{eq:measurements} can describe various low-rank matrix recovery problems including matrix completion \cite{candes2010power,gross2011recovering,keshavan2010matrix,jain2013low,sun2016guaranteed}, phase retrieval \cite{candes2015phase,candes2015phaseW}, blind deconvolution and calibration \cite{ahmed2013blind,cambareri2018through}, and sketching \cite{halko2011finding,cai2015rop,bourgain2015toward,foucart2019iterative} to name a few. There are various algorithms proposed in the literature for these problems that can be broadly categorized as follows: the algorithms based on \emph{semidefinite relaxation} \cite{candes2010power,gross2011recovering}, the iterative methods based on variants of nonconvex gradient descent \cite{keshavan2010matrix,sun2016guaranteed,ma2018implicit}, \emph{alternating minimization} \cite{jain2013low}, or approximate message passing \cite{parker2014bilinear}.
We refer the interested reader to the survey papers \cite{davenport2016overview} and \cite{chi2019nonconvex} for a broader view of the low-rank matrix recovery literature. 

Semidefinite relaxations of the low-rank matrix recovery provide the state of the art sample complexity with linear scaling in the rank and no dependence on the condition number. From the algorithmic perspective, however, these methods suffer from a high computational cost, and more importantly memory usage. This drawback motivated a suite of nonconvex approaches with comparable but weaker sample complexity guarantees \cite{chi2019nonconvex}. Another line of research, originally studied for phase retrieval, proposed recovery by a convex program that avoids lifting and the semidefinite constraint altogether \cite{bahmani2017phase,goldstein2018phasemax,bahmani2019solving,bahmani2019estimation}. This new convex approach admits inherent robustness due to convex geometry of the optimization formulation together with flexibility in adopting a number of off-the-shelf numerical convex optimization algorithms. We apply this framework to the low-rank desketching problem in this paper. We extend and streamline the analysis of the special case of rank-$1$ recovery with Gaussian factors provided in \cite{bahmani2019estimation}, to a more general low-rank recovery problem.

\subsection{Anchored regression}
Let $\tilde{\mb X}_{0}$ and $\tilde{\mb Y}_{0}$ be a pair of matrices for which $\tilde{\mb X}_{0} \tilde{\mb Y}_{0}^*$ approximates the ground truth matrix $\mb M_{0}$. With $\inp{\cdot,\cdot}$ being used throughout the paper to denote the real-valued inner product defined as 
\[
\inp{\mb U, \mb V} \defeq \mr{Re} \left( \tr(\mb U^* \mb V) \right) \,,
\]
our proposed estimator is formulated as
\begin{align}
     (\hat{\mb X}, \hat{\mb Y}) \in \argmax_{\mb X\in \mbb C^{d_1\times r}, \mb Y\in \mbb C^{d_2\times r}}\  & \inp{\tilde{\mb X}_{0},\mb X}+\inp{\tilde{\mb Y}_{0},\mb Y} - \frac{1}{n} \sum_{i=1}^n \ell_i(\mb X, \mb Y)\,, \label{eq:estimator}
\end{align}
where 
\[
\ell_i(\mb X, \mb Y) 
\defeq \frac{1}{2} \norm{\mb X^* \mb a_i}^{2} + \frac{1}{2} \norm{\mb Y^* \mb b_i}^2 + \left|\mb a_i^* \mb X \mb Y^* \mb b_i - m_i\right|\,.
\]

The optimization in \eqref{eq:estimator} is effectively a convex program and can be solved efficiently. To clarify this fact, observe that the functions $\ell_i(\mb X, \mb Y)$ can be written equivalently as
\begin{align}
\ell_i(\mb X, \mb Y)  
& = \sup_{\phi\st|\phi| = 1} \, \frac{1}{2} \norm{\mb X^* \mb a_i}^2 + \frac{1}{2} \norm{\mb Y^* \mb b_i}^2 + \mr{Re}(\phi\,(\mb a_i^* \mb X \mb Y^*\mb b_i - m_i)) \label{eq:li-1} \\
& = \sup_{\phi\st|\phi| = 1} \, \frac{1}{2} \norm{\mb X^* \mb a_i + \phi\,\mb Y^* \mb b_i}^2 - \mr{Re}(\phi\,m_i) \nonumber
\end{align}
For any fixed $\phi$, the argument of the supremum is clearly convex in $[\mb X;\ \mb Y] \in \mbb C^{(d_1+d_2) \times r}$. Therefore, $\ell_i(\mb X, \mb Y)$ is also a convex function of $[\mb X;\ \mb Y]$, meaning that \eqref{eq:estimator} is a convex program. 

Because of the specific form of the loss functions $\ell_i(\mb X,\mb Y)$, the estimator in \eqref{eq:estimator} can be viewed as a ``convexification'' of the (nonconvex) least absolute deviation (LAD) estimator by quadratic regularization. Previously, \cite{bahmani2019estimation} has studied similar estimators for observations in the form of difference of convex functions, with the bilinear observations for rank-$1$ matrices as a special case. In this paper we provide a streamlined analysis of the estimator tailored to desketching of a low-rank matrix from its rank-one measurements in \eqref{eq:measurements}. The following provides the high-level description of the sample complexity we have established for the anchored regression estimator. The precise statements are provided in Theorems \ref{thm:main_gauss} and \ref{thm:main_2design}.

\begin{theorem*}
Given a ``good'' approximation of the unknown rank-$r$ matrix $\mb M_0$ as $\tilde{\mb X}_0\tilde{\mb Y_0}^* \approx \mb M_0$,  with high probability, the proposed anchored regression recovers $\mb{M}_0$ exactly, from $O(r(d_1+d_2) \, \mr{polylog}(d_1+d_2))$ random rank-one measurements. The hidden constant in the sample complexity depends also on the ``low-rank condition number'' of $\mb M_0$.
\end{theorem*}

The sufficient number of samples for the exact recovery provided by the theorem is near optimal compared to the degrees of freedom of the rank-$r$ matrix model. The random sketching models and the size of the neighborhood will be specified in the following section. 

\subsection{Sketching Models}\label{ssec:sketching-models}
We study the sample complexity of our proposed estimator under two different random measurements models. The first model, which we refer to as the \textit{real Gaussian sketching}, simply uses the outer product of two independent Gaussian vectors as the rank-one sketching matrix. The second model, called the \textit{partially derandomized sketching}, mimics the behavior of the firs model but the random rank-one sketching matrix takes realizations from a finite set. More precisely, the first few moments of the sketching matrix of the second model are designed to coincide with those of the real Gaussian model. 

\subsubsection{Real Gaussian sketching}
This model simply considers the random vectors $[\mb a_i;\ \mb b_i] \in \mbb R^{d_1+d_2}$ to be independent copies of a random vector $[\mb a;\ \mb b] \in \mbb R^{d_1+d_2}$ satisfying
\begin{align}
    [\mb a;\ \mb b] &\sim \mc{N}\left(\mb 0,\mb I_{d_1+d_2}\right)\,. \label{eq:gaussianab}
\end{align}

\subsubsection{Partially derandomized sketching}
Next we consider a measurement distribution supported on a special finite set whose first $2t$ moments coincide with those of the real Gaussian model. 
A configuration of such a set is called a \textit{complex projective $t$-design} and designs were first introduced by Delsarte et al. 
\cite{delsarte1991spherical}. 
A variant of this model has been previously employed for the phase retrieval \cite{balan2009painless,gross2015partial,kueng2017low}. 
The concept of $t$-design has been utilized in coding theory and quantum information theory, particularly in the analysis of randomized algorithms. 
One drawback is the size of the set generally grows exponentially in the dimension while the exponent is proportional to the parameter $t$. 
Concrete constructions are widely available for the special case with degree $2$ and numerical algorithms for an approximate design for higher order are also available in the literature. 
We refer to \cite{gross2015partial} for more details on the $t$-design model and the related references. 
Below we describe the version of the model that is relevant for our purposes.

Let $P_{\mr{Sym}^t}$ denote the totally symmetric subspace of $(\mbb C^d)^{\otimes t}$ such that all elements are invariant under every possible permutation of $t$ factors (see, e.g., \cite{landsberg2012tensors}). Then a weighted $t$-design is defined as follows \cite{gross2015partial}. 

\begin{definition}
Let $t \in \mbb{N}$ and $\mb{w}_1,\dots,\mb{w}_N\in \mbb C^d$ be unit vectors. The set $\{\mb{w}_i\}_{i=1}^N$ with corresponding weights $\{p_i\}_{i=1}^N$ such that $p_i \geq 0$ for all $i=1,\dots,N$, and $\sum_{i=1}^N p_i = 1$ is a weighted complex projective $t$-design of dimension $n$ and cardinality $N$, if
\[
\sum_{i=1}^N p_i \left(\mb{w}_i \mb{w}_i^*\right)^{\otimes t} 
= \binom{d+t-1}{t}^{-1} P_{\mr{Sym}^t} \,,
\] 
where $P_{\mr{Sym}^t}$ denotes the projector onto the totally symmetric subspace $\mr{Sym}^t$ of $(\mbb C^d)^{\otimes t}$.
\end{definition}

Our second sketching model is given by the concatenation of two independent random vectors $\mb a$ and $\mb b$ in the following construction: Given a weighted $t$-design $\{(\mb w_i, p_i)\}_{i=1}^{N_1}$ in $\mbb C^{d_1}$ with $t \geq 2$, let $\mb a$ be a random vector given by
\begin{equation}
\label{eq:2design_a}
\mbb P\left\{ \mb a = \sqrt{d_1} \mb w_i \right\} = p_i, \quad i = 1,\dots,N_1 \,.
\end{equation}
Then $\mb a$ satisfies 
\begin{equation}
\label{eq:4m_2design_a}
\mbb E (\mb a \mb a^*)^{\otimes t} = d_1^t \binom{d_1+t-1}{t}^{-1} P_{\mr{Sym}^t} \,.
\end{equation}
Similarly, given a weighted $t$-design $\{(\mb w'_i, p'_i)\}_{i=1}^{N_2}$ in $\mbb C^{d_2}$ with $t \geq 2$, let $\mb b$ be a random vector given by
\begin{equation}
\label{eq:2design_b}
\mbb P\left\{ \mb b = \sqrt{d_2} \mb w'_i \right\} = p'_i, \quad i = 1,\dots,N_2 \,.
\end{equation}
Then $\mb b$ satisfies 
\begin{equation}
\label{eq:4m_2design_b}
\mbb E (\mb b \mb b^*)^{\otimes t} = d_2^t \binom{d_2+t-1}{t}^{-1} P_{\mr{Sym}^t} \,.
\end{equation}
By construction, a weighted $t$-design is automatically a weighted $t'$-design for all $t' \leq t$. Given accurate anchor matrices, a performance guarantee for \eqref{eq:estimator} can be derived only requiring the moment conditions of up to the $4$th order. However, obtaining such accurate anchor matrices require a more stringent condition on higher moments of up to order $2t = \Omega(\log(d_1+d_2))$. This is why we assume that $t \geq 2$ in constructing partially derandomized measurement vectors with $t$-designs.

\subsection{Spectral initialization}
Our main results rely on the availability of $\tilde{\mb X}_0$ and $\tilde{\mb Y}_0^*$ such that $\tilde{\mb X}_0 \tilde{\mb Y}_0^*$ is close to the ground truth matrix $\mb M_0$. To provide a stand-alone theory that does not require any oracle information, we also analyze a specific method to obtain such matrices $\tilde{\mb X}_0$ and $\tilde{\mb Y}_0^*$ described below. 

Let $\mc{A}: \mbb C^{d_1 \times d_2} \to \mbb C^n$ denote the linear operator representing the rank-one measurements in \eqref{eq:measurements}, i.e., it is defined by
\[
\mb M \mapsto 
\mc{A}(\mb{M}) = \left(\frac{1}{\sqrt{n}} \, \mb{a}_i^* \mb{M} \mb{b}_i \right)_{i=1}^n \,.
\]
Then its adjoint operator, denoted by $\mc A^*$, is given by
\[
\mb y \mapsto
\mc{A}^*(\mb y) = \frac{1}{n} \sum_{i=1}^n y_i \mb{a}_i \mb{b}_i^* \,.
\]
The \textit{spectral method} computes an estimate $\tilde{\mb M}_0$ of the unknown matrix $\mb M_0$ as the best rank-$r$ approximation of $\mc A^* \mc A (\mb{M}_0)$ with respect to the Frobenius norm. Under the two random sketching models, we obtain suitable upper bounds on the approximation error through matrix concentration inequalities. 

Next we factorize the estimated matrix into $\tilde{\mb M}_0 = \tilde{\mb X}_0 \tilde{\mb Y}_0^*$ through the singular value decomposition. Let $\tilde{\bm{M}}_0 = \tilde{\mb{U}}_0 \tilde{\mb{\varSigma}}_0 \tilde{\mb{V}}_0^*$ be the compact singular value decomposition of $\tilde{\mb M}_0$. Then we choose $\tilde{\mb X}_0$ and $\tilde{\mb Y}_0$ by
\begin{equation*}
\tilde{\mb{X}}_0 = \tilde{\mb{U}}_0 \tilde{\mb{\varSigma}}_0^{1/2}
\quad \text{and} \quad 
\tilde{\mb{Y}}_0 = \tilde{\mb{V}}_0 \tilde{\mb{\varSigma}}_0^{1/2}\,,
\end{equation*}
so that they have the same singular values. This particular decomposition provides a set of useful properties, such as the identity 
\begin{align*}
    \tilde{\mb \varSigma_0} & = \tilde{\mb X}_0^* \tilde{\mb X}_0 = \tilde{\mb Y}_0^* \tilde{\mb Y}_0\,, 
\end{align*}
that are utilized in the proof of our main results.

The precise statement of the requirements for the spectral initialization, and the corresponding sample complexity under the two considered measurement models are provided in Section \ref{sec:main_res}. The pertaining derivations are provided in Section \ref{sec:spectral-init}.

\subsection{Discussion and Related work}
Under the real Gaussian sketching model and given an initial estimate satisfying \eqref{eq:cond_init_intro}, we demonstrate that, with high probability, the estimator in \eqref{eq:estimator} recovers $\mb M_0$ exactly, provided the number of measurements scales as $n \geq C d r$, where $d = \max(d_1,d_2)$. This sample complexity coincides with the sample complexity achieved by the estimators based on lifting and semidefinite relaxation \cite{chen2015exact,cai2015rop}. On the other hand, our estimator is formulated through an explicit factorization only with $r(d_1+d_2)$ variables while the lifted convex estimator over $d_1 d_2$ variables \cite{chen2015exact,cai2015rop}. Furthermore, because the methods based on semidefinite relaxation do not operate in the factorized domain, they often need singular value calculations which further complicates their scalability.

However, computationally inexpensive methods used to find the initial estimates obeying \eqref{eq:cond_init_intro}, often lead to a suboptimal overall sample complexity. In fact, we show that, with high probability, the spectral initialization succeeds if $n \geq C d r^2$ which dominates the sample complexity $n\geq C d r$ for the ``oracle-assisted'' estimator mentioned above.

The \emph{iterative hard thresholding} algorithm is studied under variants of the restricted isometry property for low-rank matrices in \cite{jain2013low} and \cite{foucart2019iterative}. This algorithm is computationally less expensive than the generic convex optimization algorithms that solve the semidefinite relaxation, because it only requires to perform low-rank SVDs in its iterations rather than the full SVDs. However, the fact that the iterative hard thresholding method operates in the lifted domain, is an obstruction to its scalability.  Several other iterative algorithms have been proposed and analyzed under the real Gaussian sketching model. Earlier methods used \textit{resampling} to draw fresh measurements per iteration. Therefore, these methods need to terminate after finitely many iterations, which only allows for approximate recovery up to a prescribed accuracy $\epsilon$. Prior work on this approach achieve the sample complexities $O(d r^4 \log^2 d \log(1/\epsilon))$ \cite{zhong2015efficient}, $O(d r^3 \log(1/\epsilon))$ \cite{lin2016non}, and $O(d r^2 \log^4 d \log(1/\epsilon))$ \cite{soltani2017improved}. In more recent work, \cite{sanghavi2017local} and \cite{li2018nonconvex} studied performance of the nonconvex gradient descent and established the sample complexities $O(d r^6 \log^2 d)$ and $O(d r^4 \log d)$, respectively. Our estimator outperforms these results for nonconvex approaches. In fact, our estimator would have achieved the ideal sample complexity should there be an initialization with the sample complexity $O(dr)$. The hidden constant in this sample complexity, similar to the sample complexity of the nonconvex methods, depends on the ``low-rank condition number'' of the ground truth matrix, defined precisely below in Section \ref{sec:main_res}.  
It is also worthwhile to mention that the existing results in the literature often focus on the case where the rank-one measurement matrices or the ground truth matrix are symmetric. The model \eqref{eq:measurements} considered in this paper allows for a general choice of the measurement factors $\mb a_i$ and $\mb b_i$. For simplicity, we only consider independent factors $\mb a_i$ and $\mb b_i$, but the provided framework can be adapted to the case of dependent factors by modifying some of the relevant calculations. 
\section{Main results}
\label{sec:main_res}
Our main results demonstrate how many observations suffice for the estimator \eqref{eq:estimator} to reconstruct the unknown matrix $\mb X_0 \mb Y_0^*$. Our first theorem provides a sample complexity that guarantees accuracy of the estimator \eqref{eq:estimator} under the real Gaussian sketching model. Throughout we use $\kappa\ge 1$ to denote the (low-rank) condition number of $\mb M_0$, which refers to the ratio of the largest and smallest \emph{non-zero} singular values of $\mb M_0$, i.e.,
\[
    \kappa = \frac{\sigma_1(\mb M_0)}{\sigma_r(\mb M_0)}\,.
\]

\begin{theorem}[Real Gaussian desketching]
\label{thm:main_gauss}
Let $([\mb{a}_i;\ \mb{b}_i])_{i=1}^n$ be independent copies of $[\mb{a};\ \mb{b}] \sim \mc N(\mb 0, \mb I_{d_1+d_2})$. Let $\tilde{\mb X}_0\in \mbb C^{d_1 \times r}$ and $\tilde{\mb Y}_0 \in \mbb C^{d_2 \times r}$ be matrices that satisfy $\tilde{\mb X}_0^*\tilde{\mb X}_0=\tilde{\mb Y}^*_0\tilde{\mb Y}_0$ and 
\begin{equation}
\label{eq:cond_init_intro}
\norm{\tilde{\mb{X}}_0\tilde{\mb{Y}}^*_0 - \mb{M}_0} 
\lesssim r^{-1/2} \kappa^{-2} \|\mb M_0\| \,.
\end{equation}
If the number of measurements $n$ obeys
\begin{equation}
\label{eq:sampcomp_gauss}
n \gtrsim \max\{\kappa r (d_1+d_2), \log(1/\delta)\}\,, 
\end{equation}
then with probability at least $1-\delta$ the estimates $\hat{\mb X}$ and $\hat{\mb Y}$ obtained by the anchored regression  satisfy $\hat{\mb X}\hat{\mb Y}^*=\mb M_0$.
\end{theorem}

The result by Theorem~\ref{thm:main_gauss} is comparable to the analogous result for the lifted convex optimization by nuclear norm minimization \cite{chen2015exact,cai2015rop}. However, the dependence on the condition number, which does not appear in the lifted case, is the cost we need to pay to save the computation through explicit factorization. 

We also present the sample complexity for the success of the spectral initialization under the same model. 

\begin{prop}
\label{prop:init_gauss}
Let $([\mb{a}_i;\ \mb{b}_i])_{i=1}^n$ be independent copies of $[\mb{a};\ \mb{b}] \sim \mc N(\mb 0, \mb I_{d_1+d_2})$. Then the estimate $(\tilde{\mb X}_0,\tilde{\mb Y}_0)$ by the spectral initialization satisfies \eqref{eq:cond_init_intro} with probability at least $1-(d_1+d_2)^{-\alpha}$, if
\[
n \gtrsim \alpha^3 \kappa^4 r^2 (d_1+d_2) \log^3 (d_1+d_2)\,.
\]
\end{prop}

As shown in Theorem~\ref{thm:main_gauss} and Proposition~\ref{prop:init_gauss}, the number of samples enough for the success of the spectral initialization dominates that for the estimator. Although the spectral initialization is just one approach to obtain an initial estimate satisfying \eqref{eq:cond_init_intro}, it has not been shown any alternative practical method providing the same accuracy from fewer measurements. 

Next we present the corresponding results for partially derandomized sketching below. 

\begin{theorem}[Partially derandomized desketching]
\label{thm:main_2design}
Let $\mb a$ and $\mb b$ are independent random vectors uniformly distributed over the corresponding $t$-design sets according to \eqref{eq:2design_a} and \eqref{eq:2design_b} with $t \geq 2$. Let $([\mb{a}_i;\ \mb{b}_i])_{i=1}^n$ be independent copies of $[\mb{a};\ \mb{b}]$. Let $\tilde{\mb M}_0$, $\tilde{\mb X}_0$, and $\tilde{\mb Y}_0$ as in Theorem~\ref{thm:main_gauss} satisfying \eqref{eq:cond_init_intro}. If the number of measurements $n$ satisfies
\begin{equation}
\label{eq:sampcomp_2design}
n \gtrsim \kappa r (d_1+d_2) \max\{\log (d_1+d_2), \log(1/\delta)\}\,,
\end{equation}
then the anchored regression exactly recovers $\mb M_0$ with probability at least $1-\delta$. 
\end{theorem}

\begin{prop}
\label{prop:init_2design}
Under the sketching model in Theorem~\ref{thm:main_2design}, the spectral initialization provides $(\tilde{\mb X}_0,\tilde{\mb Y}_0)$ satisfying \eqref{eq:cond_init_intro} with probability at least $1-(d_1+d_2)^{-\alpha}$ provided 
\begin{equation}
\label{eq:suff_n_td}
n \gtrsim 
t \kappa^4 r^2 (d_1+d_2)^{1+1/2t+\alpha/t} 
\vee 
\left(
t^2 \kappa^2 r^{3/2} (d_1+d_2)^{1+\alpha/2t}
\right)^{2t/(2t-1)}\,.
\end{equation}
Particularly, if $t \geq (\alpha+1/2) \ln(d_1+d_2)$, then the condition in \eqref{eq:suff_n_td} simplifies to 
\[
n \gtrsim \alpha^2 \kappa^4 r^2 (d_1+d_2) \ln^4(d_1+d_2)\,.
\]
\end{prop}

Compared to the real Gaussian sketching model, the derandomized case is guaranteed by slightly more measurements (larger by a logarithmic factor). 

\section{Numerical Results}

A set of Monte Carlo numerical results are provided to illustrate that the empirical behavior of the estimator is consistent with the main theoretical results. 
We first discuss how the convex program in \eqref{eq:estimator} can be solved by a practical numerical algorithm. 
Recall that the estimator in \eqref{eq:estimator} is equivalently written as
\begin{align}
    \label{eq:estimator2}
     (\hat{\mb X}, \hat{\mb Y}) \in \argmin_{\mb X\in \mbb C^{d_1\times r}, \mb Y\in \mbb C^{d_2\times r}}\  & f (\mb X, \mb Y)\,,
\end{align}
where the convex objective function is given by
\[
f (\mb X, \mb Y) = - \inp{\tilde{\mb X}_{0},\mb X} - \inp{\tilde{\mb Y}_{0},\mb Y} - \frac{1}{n} \sum_{i=1}^n \left( \frac{1}{2} \norm{\mb X^* \mb a_i}^{2} + \frac{1}{2} \norm{\mb Y^* \mb b_i}^2 + \left|\mb a_i^* \mb X \mb Y^* \mb b_i - m_i\right| \right)\,.
\]

A simple \emph{subgradient method} can be used to find a minimizer to \eqref{eq:estimator2}. 
The estimator is refined by
\[
\begin{bmatrix} \mb X_{t+1} \\ \mb Y_{t+1} \end{bmatrix} = \begin{bmatrix} \mb X_t \\ \mb Y_t \end{bmatrix} - \eta_t \mb G_t\,,
\]
where $\eta_t$ denotes the step size at the $t$th iteration and $\mb G_t \in \mbb C^{(d_1+d_2) \times r}$ is a subgradient of $f$ at $[\mb X_t \; \mb Y_t]$ specifically given by 
\[
\mb G_t = - \begin{bmatrix} \tilde{\mb X}_0 \\ \tilde{\mb Y}_0 \end{bmatrix} + \frac{1}{n} \sum_{i=1}^n \begin{bmatrix} \mb a_i \mb a_i^* & \phi_i^* \mb a_i \mb b_i^* \\ \phi_i \mb b_i \mb a_i^* & \mb b_i \mb b_i^* \end{bmatrix} \begin{bmatrix} \mb X_t \\ \mb Y_t \end{bmatrix}\,,
\]
where
\[
\phi_i = \exp\left(\imath \Arg(\mb a_i^* \mb X_t \mb Y_t^* \mb b_i - m_i) \right)\,,
\]
with $\imath=\sqrt{-1}$ denoting imaginary unit, and $\Arg\left(z\right)$ denoting the principal argument of $z\in \mbb C$. 
The per-iteration-cost of this subgradient method is comparable to that of the nonconvex gradient descent. 
We use the diminishing step size rule for $(\eta_t)_t$. We chose the simplest algorithm to solve the convex program in \eqref{eq:estimator2}. However, we believe that more sophisticated optimization algorithms are also applicable to our problem. For example, in expression of $\ell_i(\cdot,\cdot)$ as \eqref{eq:li-1} the constraint $|\phi|=1$ can be relaxed to $|\phi|\le 1$ without affecting the function value. Then, we can show that the proposed estimator can be equivalently formulated as a convex-concave saddle-point problem, which can be solved using algorithms based on \emph{mirror descent} and \emph{mirror prox} \cite[Chs. 5 and 6]{sra2011optimization}, \cite[Ch. 5]{bubeck2015convex}.

The first set of simulation provides the empirical phase transition as a function of the rank while the other parameters are fixed ($d_1 = d_2 = 128$ and $n = d_1 d_2 /4$). The measurements are generated by the standard complex normal distribution. 
Theorem~\ref{thm:main_gauss} shows that, if $\tilde{\mb X_0} \tilde{\mb Y_0}^*$ is close enough to the ground-truth matrix $\mb M_0$, the maximum rank that leads to the exact recovery is determined by \eqref{eq:sampcomp_gauss}. In order to consider the effect of the accuracy of the anchor, we introduce a parameter $\alpha \in [0,1]$ to linearly interpolate between the spectral initialization, corresponding to $\alpha =1$, and the ground truth, corresponding to $\alpha=0$. Figure~\ref{fig:empPT_noiseless} illustrates the empirical phase transitions, with the left and right panels respectively showing the median and $90$th percentile of the relative error over $100$ trials. Exact recovery is achieved in most cases for $r\le 2$. The top rows in Figure~\ref{fig:empPT_noiseless} shows the empirical recovery phase transition when $(\tilde{\mb X_0}, \tilde{\mb Y_0})$ is given by the spectral initialization, whereas the bottom rows correspond to the ground truth chosen as the anchor. Figure~\ref{fig:empPT_noiseless} suggests that at larger $r$ the phase transition occurs at smaller $\alpha$, thereby requiring a more accurate initial estimate $\tilde{\mb X_0} \tilde{\mb Y_0}^*$, which is consistent with the requirement in \eqref{eq:cond_init_intro}. 

\begin{figure}[htb]
    \noindent
    \begin{subfigure}[b]{0.49\textwidth}
        \centering
        \includegraphics[scale=0.4]{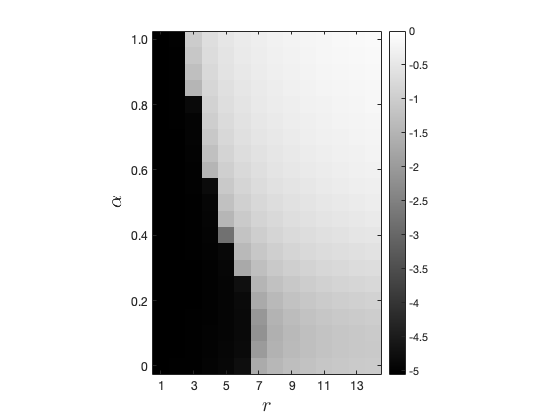}
        \caption{Median}
    \end{subfigure}
    \hfill
    \begin{subfigure}[b]{0.49\textwidth}
        \centering
        \includegraphics[scale=0.4]{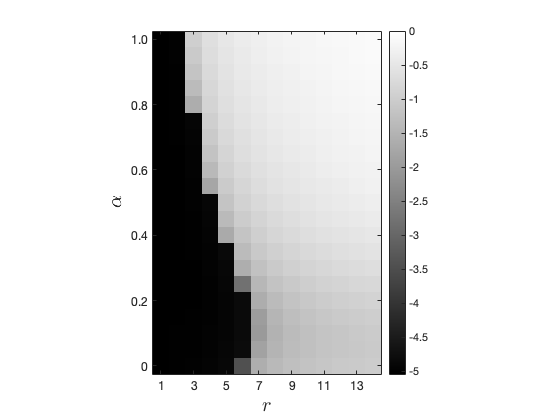}
        \caption{$90$th percentile}
    \end{subfigure}
    \caption{Empirical phase transition in the noiseless case. The logarithm base $10$ of percentiles of the normalized estimation error $\|\hat{\mb X} \hat{\mb Y}^* - \mb M_0\|_\mathrm{F}/\|\mb M_0\|_\mathrm{F}$ is plotted. The size of matrix and the number of measurements are set to $d_1 = d_2 = 128$ and $n = d_1 d_2 /4$, respectively. The anchor matrix is computed from a convex combination of the rank-$r$ matrix in the spectral initialization and the ground truth respectively with weights $\alpha$ and $1-\alpha$. Thus $\alpha = 1$ denotes the case with the spectral initialization.}
    \label{fig:empPT_noiseless}
\end{figure}

The second simulation illustrates how the additive noise to measurements propagates to the estimation error. We consider the regime of parameters where the convex estimator provides exact recovery in the noiseless case. We set $d_1 = d_2 = 128$, $r = 2$, and $n = d_1 d_2/4$. The $90$th percentile of $100$ realizations was observed while the signal-to-noise ratio (SNR) varies over $5$ to $50$ dB.  
Figure~\ref{fig:empPT_wnoise} shows that the estimation error decreases gradually as SNR increases. 
In other words, the estimation error in the presence of measurement noise scales smoothly as a function of SNR. 
This stable performance of the estimator is due to the nice geometry of the convex program in \eqref{eq:estimator2}.

\begin{figure}[htb]
	\centering
    \includegraphics[scale=0.3]{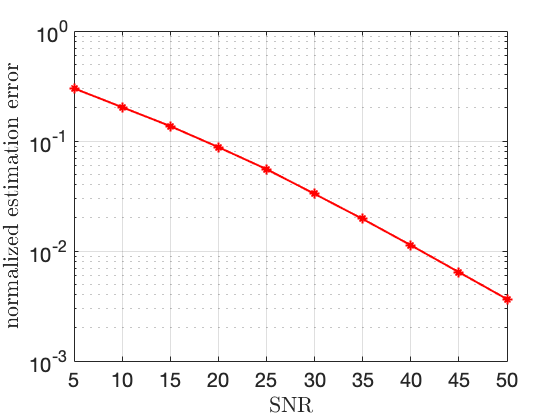}
    \caption{Estimation error for varying SNR. The observation is corrupted with additive Gaussian noise so that $m_i = \mb a_i^* \mb M_{0}\mb b_i + \zeta_i$ for $i=1,\dots,n$ with $\zeta_1,\dots,\zeta_n$ being i.i.d. $\mathcal{N}(0,\sigma^2)$. SNR is defined as $10 \log_{10}(\sum_{i=1}^n m_i^2/\sigma^2)$. The $90$th percentile of the normalized estimation error is plotted. The size of matrix, rank, and number of measurements are set to $d_1 = d_2 = 128$, $r=2$, and $n = d_1 d_2 /4$, respectively. The anchor matrix is computed by the spectral initialization.}
    \label{fig:empPT_wnoise}
\end{figure}

\section{Proof of the Main Theorems}
We prove Theorems~\ref{thm:main_gauss} and \ref{thm:main_2design} in two steps. First, we derive a set of deterministic conditions to guarantee exact recovery for the proposed estimator. Then, we show that these conditions hold, with high probability, for the sketching models introduced in Section \ref{ssec:sketching-models}.

\subsection{A deterministic sufficient condition for exact recovery}

Recall that $\mb M_0 = \mb U_0 \mb \varSigma_0 \mb V_0^*$ is the compact singular decomposition of the ground truth matrix $\mb M_0$, where both $\mb U_0 \in \mbb C^{d_1 \times r}$ and $\mb V_0 \in \mbb C^{d_2 \times r}$ have orthonormal columns, and $\mb{\varSigma}_0 \in \mbb R^{r \times r}$ is the diagonal matrix of the singular values. The support space of $\mb M_0$, denoted by $\mc T$, is defined as
\begin{align*}
    \mc T & \defeq\left\{ \mb{\varDelta}_{1} \mb V_{0}^* + \mb U_0 \mb{\varDelta}_{2}^* \,:\,\mb{\varDelta}_i\in\mbb C^{d_i\times r}, \, i=1,2\right\}\,.
\end{align*} With these notations, the following proposition provides a set of deterministic sufficient conditions for the estimator \eqref{eq:estimator} to exactly recover $\mb M_0$. The proof is deferred to Section~\ref{sec:proof:prop:det_suff}.

\begin{prop}
\label{prop:det_suff}
Let $r$ and $\kappa$ respectively denote the rank and the condition number of $\bm M_0$, and $(\tilde{\mb X}_0, \tilde{\mb Y}_0)$ be a given pair of matrices that satisfy $\tilde{\mb X}_0^*\tilde{\mb X}_0=\tilde{\mb Y}^*_0\tilde{\mb Y}_0$. Furthermore, suppose that there exist $\rho \in (0,1]$, and absolute constants $C_1, C_2 \geq 1$ such that
\begin{align}
\label{eq:cond_smallball}
\frac{1}{n}\sum_{i=1}^n \left|\inp{\mb a_i \mb b_i^*, \mb{H}}\right|
\ge \rho \norm{\mb H}_\F, \quad \text{for all}\ \mb{H} \in \mc{T}\,,
\\
\label{eq:isometryA}
\norm{\mb{I}_{d_1} - \frac{1}{n} \sum_{k=1}^n \mb{a}_k \mb{a}_k^*} 
\leq \frac{\rho}{C_1 \sqrt{r \kappa}} \,,
\\
\label{eq:isometryB}
\norm{\mb{I}_{d_2} - \frac{1}{n} \sum_{k=1}^n \mb{b}_k \mb{b}_k^*} 
\leq \frac{\rho}{C_1 \sqrt{r \kappa}} \,,
\intertext{and}
\label{eq:cond_init}
\norm{\tilde{\mb X}_0 \tilde{\mb Y}_0^* - \mb{M}_0} 
\leq \frac{\rho \norm{\mb{M}_0}}{C_2\sqrt{r} \kappa^2} \,.
\end{align}
Then, the maximizer $(\hat{\mb{X}}, \hat{\mb{Y}})$ to \eqref{eq:estimator} is unique and satisfies $\hat{\mb{X}} \hat{\mb{Y}}^* = \mb{M}_0$. 
\end{prop}

\subsection{Verifying the sufficient condition under the random models}

We demonstrate that under the random measurement models introduced in Section~\ref{sec:main_res} the assumptions made in Proposition~\ref{prop:det_suff} hold with high probability. 

\subsubsection{Small-ball method}

Our analysis is based on the \textit{small-ball method} \cite{koltchinskii2015bounding,mendelson2014learning}. A simple exposition and some applications of this method can be found in \cite{ tropp2015convex} and \cite{dirksen2016gap}. We provide the proofs for the manuscript to be self-contained as well as addressing some subtle but important differences due to operation in the complex domain.

We first show in the following proposition that \eqref{eq:cond_smallball} is satisfied with high probability.
\begin{prop}[Lower-tail via the small-ball method]\label{prop:small-ball}
Let $\mc T$ be a subset of $\mbb C^{d_1 \times d_2}$ that is invariant under multiplication by unit-modulus scalars. For i.i.d. and isotropic random vectors $[\mb{a}; \ \mb{b}]$, $[\mb{a}_1;\ \mb{b}_1]$, \ldots, $[\mb{a}_n;\ \mb{b}_n]\in \mbb{C}^{d_1 + d_2}$ define
\begin{align*}
p_\tau(\mc T) 
&\defeq \inf_{\mb H \in \mc T \setminus \{\mb 0\}} 
\mbb P \left\{ \left|\mb a^*\mb H\mb b\right| \geq \tau \norm{\mb H}_\F \right\} \,, \\
\mfk C_n(\mc T) &\defeq \mbb E \sup_{\mb H \in \mc T \setminus \{\mb 0\}} \frac{1}{\sqrt{n}} \sum_{i=1}^n \frac{\varepsilon_i \inp{\mb a_i \mb b_i^*, \mb H}}{\norm{\mb H}_\F} \,,
\end{align*}
where $\varepsilon_1,\dotsc,\varepsilon_n$ are i.i.d. Rademacher random variables independent of everything else. Then, for any $\tau > 0$ and $\delta \in (0,1)$, with probability at least $1 - \delta$, we have
\begin{equation}
\inf_{\mb H \in \mc T \setminus \{\mb 0\}} \frac{1}{n}\sum_{i=1}^n \frac{\left|\mb a_i^* \mb{H} \mb b_i\right|}{\norm{\mb H}_\F}
\ge \tau p_\tau(\mc T) - \frac{\pi \, \mfk{C}_n(\mc T)}{\sqrt{n}} - \tau \sqrt{\frac{\log(1/\delta)}{n}} \,.
\label{eq:master}
\end{equation}
\end{prop}
\begin{proof}
Using $[z]_{\le t} \defeq \min\{z,t\}$ to denote the ``saturation'' at $t$, for any $\tau>0$, we have
\begin{align}
    \frac{1}{n}\sum_{i=1}^{n}\left|\mb a_i^*\mb H\mb b_i\right|  &\ge \frac{1}{n}\sum_{i=1}^{n}{\left[\left|\mb a_i^*\mb H\mb b_i\right|\right]}_{\le \tau \norm{\mb H}_\F}\,, \label{eq:truncation}
\end{align}
for every $\mb{H}\in\mbb{R}^{d_1\times d_2}$. By normalizing by $\norm{\mb{H}}_\F$, it suffices to find a lower bound for the right-hand side of \eqref{eq:truncation} for all $\mb{H}\in\mc{T}\bigcap \mbb{S}$, where $\mbb{S}$ denotes the unit sphere of the Frobenius norm in $\mbb C^{d_1\times d_2}$. 

Adding and subtracting $\E\left({\left|\mb a_i^*\mb H\mb b_i\right|}_{\le \tau }\right)$, and using the fact that $\E\left({\left[\left|\mb a_i^*\mb H\mb b_i\right|\right]}_{\le \tau}\right) \ge \tau \, \P\left(\left|\mb a_i^*\mb H\mb b_i\right|\ge \tau \right)$, we can write
\begin{equation}
    \begin{aligned}
        \frac{1}{n}\sum_{i=1}^{n}{\left[\left|\mb a_i^*\mb H\mb b_i\right|\right]}_{\le \tau} 
        & \ge \frac{1}{n}\sum_{i=1}^n \tau \, \P\left(\left|\mb a_i^*\mb H\mb b_i\right|\ge \tau \right)    \\
        &\hphantom{\ge} - \frac{1}{n}\sum_{i=1}^n \E\left({\left[\left|\mb a_i^*\mb H\mb b_i\right|\right]}_{\le \tau}\right) - {\left[\left|\mb a_i^*\mb H\mb b_i\right|\right]}_{\le \tau} \,. 
    \end{aligned} \label{eq:mean-var-decomposition}
\end{equation}

The function $F:(\mbb{C}^d)^n\to \mbb {R}_{\ge 0}$ defined as
\begin{align*}
  F\left([\mb{a}_1;\ \mb{b}_1],\dotsc,[\mb{a}_n;\ \mb{b}_n]\right)
  & \defeq \sup_{\mb{H} \in\mc{T}\bigcap\mbb{S}} \frac{1}{n} 
  \left|\sum_{i=1}^n 
  \E\left({\left[\left|\mb a_i^*\mb H\mb b_i\right|\right]}_{\le \tau}\right) - {\left[\left|\mb a_i^*\mb H\mb b_i\right|\right]}_{\le \tau}
  \right|\,
\end{align*}
has the bounded difference property. Therefore, invoking the \emph{bounded difference inequality} \cite{mcdiarmid1989method}, 
with probability at least $1-\delta$, we have
\begin{align}
  F\left([\mb{a}_1;\ \mb{b}_1],\dotsc,[\mb{a}_n;\ \mb{b}_n]\right) & \le \E F\left([\mb{a}_1;\ \mb{b}_1],\dotsc,[\mb{a}_n;\ \mb{b}_n]\right) + \tau \, \sqrt{\frac{\log(1/\delta)}{2n}}\,. \label{eq:F-EF}
\end{align}
Using the standard Gin\'{e}-Zinn symmetrization argument (e.g., see \cite[Lemma 2.3.1]{vDVW96}), the expectation on the right-hand side of \eqref{eq:F-EF} can be bounded as
\begin{align*}
  \E F\left([\mb{a}_1;\ \mb{b}_1],\dotsc,[\mb{a}_n;\ \mb{b}_n]\right) & \le \frac{2}{\sqrt{n}} \, \E\left(\sup_{\mb{H}\in\mc{T}\bigcap\mbb{S}}\frac{1}{\sqrt{n}}\sum_{i=1}^n \varepsilon_i{\left[\left|\mb a_i^*\mb H\mb b_i\right|\right]}_{\le \tau}\right)\,, 
\end{align*}
where the expectation on the right-hand side is with respect to $[\mb{a}_1; \ \mb{b}_1],\dotsc,[\mb{a}_n; \ \mb{b}_n]$ as well as the i.i.d. Rademacher random variables $\varepsilon_1,\dotsc,\varepsilon_n$. Since the function $z\mapsto {\left[\left|z\right|\right]}_{\le \tau }$ is $1$-Lipschitz, invoking the \emph{Rademacher contraction principle} \cite[Theorem 4.12]{ledoux1991probability} yields 
 \begin{align}
  \E\left(\sup_{\mb{H}\in\mc{T}\bigcap\mbb{S}}\frac{1}{\sqrt{n}}\sum_{i=1}^n \varepsilon_i{\left[\left|\mb a_i^*\mb H\mb b_i\right|\right]}_{\le \tau}\right) 
  & \le \,\E\left(\sup_{\mb{H}\in\mc{T}\bigcap\mbb{S}}\frac{1}{\sqrt{n}}\sum_{i=1}^n  \varepsilon_i \left|\mb a_i^*\mb H\mb b_i\right|\right) \,. \label{eq:contraction}
\end{align}
Let $\phi$ be a unit-modulus scalar in $\mbb C$ that is selected uniformly at random, and $\E_\phi$ denote the expectation with respect to $\phi$ conditioned on everything else. Straightforward calculus shows that for any $z\in \mbb C$ we have $|z|=(\pi/2)\E_\phi\left(|\mr{Re}(\phi^*z)|\right)$. Applying this identity in \eqref{eq:contraction} then yields
\begin{align*}
  \E\left(\sup_{\mb{H}\in\mc{T}\bigcap\mbb{S}}\frac{1}{\sqrt{n}}\sum_{i=1}^n \varepsilon_i{\left[\left|\mb a_i^*\mb H\mb b_i\right|\right]}_{\le \tau}\right) 
  & \le \frac{\pi}{2}\,\E\left(\sup_{\mb{H}\in\mc{T}\bigcap\mbb{S}}\frac{1}{\sqrt{n}}\sum_{i=1}^n  \varepsilon_i \E_\phi\left|\inp{\mb a_i\mb b_i^*,\phi \mb H}\right|\right) \\
  & \le \frac{\pi}{2}\,\E\E_\phi\left(\sup_{\mb{H}\in\mc{T}\bigcap\mbb{S}}\frac{1}{\sqrt{n}}\sum_{i=1}^n  \varepsilon_i \left|\inp{\mb a_i\mb b_i^*,\phi \mb H}\right|\right) \\
  & \le \frac{\pi}{2}\,\E\left(\sup_{\mb{H}\in\mc{T}\bigcap\mbb{S}}\frac{1}{\sqrt{n}}\sum_{i=1}^n  \varepsilon_i \left|\inp{\mb a_i\mb b_i^*,\mb H}\right|\right)\\
  &\le \frac{\pi}{2}\,\mfk C_n(\mc T)\,,
\end{align*}
where the second, third, and fourth lines follow respectively from the Jensen's inequality, the assumption that $\mc T$ is invariant under multiplication by unit-modulus scalars, and applying the Rademacher contraction principle once more.

Furthermore, since $[\mb{a}_1;\ \mb{b}_1],\dotsc,[\mb{a}_n;\ \mb{b}_n]\in \mbb{C}^{d_1 + d_2}$ are identically distributed, we have
\begin{equation}
\frac{1}{n}\sum_{i=1}^n \P\left(\left|\mb a_i^*\mb H\mb b_i\right|\ge \tau \right)
= \P\left(\left|\mb a^*\mb H\mb b\right|\ge \tau \right) \,.
\label{eq:avg_prob}
\end{equation}

Therefore, in view of \eqref{eq:mean-var-decomposition}, \eqref{eq:F-EF}, and \eqref{eq:avg_prob}, with probability at least $1-\delta$, for all $\mb H\in \mc T\cap \mbb S$ we have
    \begin{align*}
        \frac{1}{n}\sum_{i=1}^{n}{\left[\left|\mb a_i^*\mb H\mb b_i\right|\right]}_{\le \tau} 
        & \ge \tau  \, \P\left(\left|\mb a^*\mb H\mb b\right|\ge \tau \right) -\frac{\pi\,\mfk{C}_n(\mc T)}{\sqrt{n}} - \tau \, \sqrt{\frac{\log(1/\delta)}{2n}}\,.
    \end{align*} 
Recalling the definition of $p_\tau(\mc T)$ is enough to complete the proof. 
\end{proof}

We apply Proposition~\ref{prop:small-ball} under the assumptions in either of Theorems~\ref{thm:main_gauss} and \ref{thm:main_2design}. Then \eqref{eq:cond_smallball} is satisfied with high probability provided that the right-hand side of \eqref{eq:master} is lower bounded by a nonnegative scalar $\rho$. The following lemmas provides a lower (resp. upper) bound on $p_\tau(\mc T)$ (resp. $\mfk{C}_n(\mc T)$). The proofs are provided in Appendix sections \ref{sec:proof:lem:lbprob} and \ref{sec:proof:lem:ubradcpl}.

\begin{lem}[Lower bound on probability]
\label{lem:lbprob}
Let $[\mb a;\ \mb b]$ to be a random vector drawn either according to \eqref{eq:gaussianab}, or the pair \eqref{eq:2design_a} and \eqref{eq:2design_b}. Then 
\[
p_\tau(\mc T) \geq c (1-\tau^2)^2
\]
for an absolute constant $c > 0$. 
\end{lem}

\begin{lem}[Upper bound on Rademacher complexity]
\label{lem:ubradcpl}
Let $[\mb a;\ \mb b]$ satisfy that i) $\mb a $ and $\mb b$ are independent; ii) each of $\mb a$ and $\mb b$ is isotropic. Then 
\[
\mfk{C}_n(\mc T) \leq \sqrt{(d_1 + d_2) r} \,,
\]
\end{lem}

By plugging in the results of the above lemmas, a sufficient condition for satisfying \eqref{eq:cond_smallball} with probability $1-\delta$ is given by
\begin{equation}
\label{eq:cond_smallball_suff}
c (1-\tau^2)^2 - \frac{4 \sqrt{(d_1 + d_2) r}}{\sqrt{n}} - \tau \sqrt{\frac{\log(1/\delta)}{2n}} \geq \rho \,.
\end{equation}

Given $\rho$, by choosing $C$ in the assumption $n \geq C \max\{r(d_1+d_2), \log(1/\delta)\}$ large enough and by choosing $\tau$ small we obtain that \eqref{eq:cond_smallball_suff} holds with probability $1-\delta$.

\subsubsection{Approximate Isotropy}

Next we show that \eqref{eq:isometryA} and \eqref{eq:isometryB} are satisfied with high probability for both the Gaussian and $t$-design cases. To simplify the notation, let $\eta$ denote the right-hand side of \eqref{eq:isometryA}, which coincides with that of \eqref{eq:isometryB}, i.e., $\eta = \rho/C_1\sqrt{r\kappa}$. We are interested in the regime where $0 < \eta < 1$. 

In the Gaussian case, the concentration of extreme singular values of a Wishart matrix has been well studied in the literature (e.g., see \cite[Theorem~II.13]{davidson2001local}, which is summarized as Theorem~\ref{thm:wishart} in Appendix). If $[\mb a ;\ \mb b]$ is a standard Gaussian random vector, then we have
\[
\mbb{P}\left\{
\left\|\frac{1}{n} \sum_{i=1}^n \mb a_i \mb a_i^* - \mb{I}_{d_1}\right\| 
> 3 \max\left( \sqrt{\frac{4d_1}{n}}, \frac{4d_1}{n} \right)
    \right\}
\leq 2 \exp(-d_1/2) \,
\]
and
\[
\mbb{P}\left\{
\left\|\frac{1}{n} \sum_{i=1}^n \mb b_i \mb b_i^* - \mb{I}_{d_2}\right\| 
> 3 \max\left( \sqrt{\frac{4d_2}{n}}, \frac{4d_2}{n} \right)
    \right\}
\leq 2 \exp(-d_2/2) \,.
\]
Therefore, \eqref{eq:isometryA} and \eqref{eq:isometryB} are satisfied with probability $1-\delta$ if
\[
n \geq \max\{36 \eta^{-2} (d_1+d_2), 2 \log(4/\delta)\} 
= \max\{C_1^2 \rho^{-2} \kappa r (d_1+d_2), 2 \log(4/\delta)\}\,,
\]
which is implied by \eqref{eq:sampcomp_gauss} in Theorem~\ref{thm:main_gauss}.

In the $t$-design case with $t \geq 2$, we obtain a similar result via the matrix Bernstein inequality \cite[Theorem~1.6]{tropp2012user}, summarized as Theorem~\ref{thm:matberntropp} in the appendix. If $[\mb a ;\ \mb b]$ satisfy \eqref{eq:2design_a} and \eqref{eq:2design_b}, then we have
\[
\mbb{P}\left\{
\left\|\frac{1}{n} \sum_{i=1}^n \mb a_i \mb a_i^* - \mb{I}_{d_1}\right\| > \eta
    \right\}
\leq 2 d_1 \exp\left( \frac{-\eta^2 n}{4 d_1} \right) \,
\]
and
\[
\mbb{P}\left\{
\left\|\frac{1}{n} \sum_{i=1}^n \mb b_i \mb b_i^* - \mb{I}_{d_2}\right\| > \eta
    \right\}
\leq 2 d_2 \exp\left( \frac{-\eta^2 n}{4 d_2} \right) \,.
\]
Therefore, \eqref{eq:isometryA} and \eqref{eq:isometryB} are satisfied with probability $1-\delta$ if
\[
n \geq \frac{4 (d_1+d_2)}{\eta^2} \cdot \log\left( \frac{4(d_1+d_2)}{\delta} \right)
= \frac{4 C_1^2 \kappa r(d_1+d_2)}{\rho^2} \cdot \log\left( \frac{4(d_1+d_2)}{\delta} \right)\,,
\]
which is implied by \eqref{eq:sampcomp_2design} in Theorem~\ref{thm:main_2design}.

\section{Proof of Proposition~\ref{prop:det_suff}}
\label{sec:proof:prop:det_suff}

For conciseness, we introduce the following shorthand notations. Let 
\begin{align*}
\mb A &= \bmx{\mb a_1 & \mb a_2 & \dots & \mb a_n}\,, & \mb B &= \bmx{\mb b_1 & \mb b_2 & \dots & \mb b_n} \,, 
\end{align*}
 and define
\begin{equation}
\label{eq:defeta}
\eta \defeq \frac{\rho}{C_1 \sqrt{r \kappa}}. 
\end{equation}
Then, \eqref{eq:isometryA} and \eqref{eq:isometryB} are equivalent to 
\begin{equation}
\label{eq:isometry}
\norm{\mb{I}_{d_1} - \frac{1}{n} \mb A \mb A^*} \leq \eta 
\quad \text{and} \quad
\norm{\mb{I}_{d_2} - \frac{1}{n} \mb B \mb B^*} \leq \eta \,.
\end{equation}

First, through the following lemma we establish a sufficient optimality condition needed to prove Proposition~\ref{prop:det_suff}. 

\begin{lem}\label{lem:sufficient condition}
Let $\mb X_0 \in \mbb{C}^{n_1 \times r}$ and $\mb Y_0 \in \mbb{C}^{n_2 \times r}$ satisfy $\mb X_0 \mb Y_0^* = \mb M_0$. Then $[\mb X_0 ;\ \mb Y_0]$ is the unique maximizer of \eqref{eq:estimator}, if for every $\mb{\varDelta}_1\in\mbb{C}^{d_1\times r}$ and $\mb{\varDelta}_2\in\mbb{C}^{d_2\times r}$ we have 
\begin{equation}
\label{eq:sufficient condition}
\inp{\tilde{\mb X}_0 - \frac{1}{n} \mb A \mb A^* \mb{X}_0,\mb \varDelta_1} +\inp{\tilde{\mb Y}_0 - \frac{1}{n} \mb B \mb B^* \mb{Y}_0,\mb\varDelta_2} 
\leq \frac{1}{n}\sum_{i=1}^n\left|\inp{\mb{a}_i\mb{b}_i^*,\mb{X}_0\mb\varDelta_2^*+\mb\varDelta_1\mb Y_0^*}\right|
\end{equation}
with equality occurring only when both $\mb \varDelta_1$ and $\mb \varDelta_2$ are zero.
\end{lem}

\begin{proof}
Let $\mb X_0 \in \mbb C^{d_1 \times r}$ and $\mb Y_0 \in \mbb C^{d_2 \times r}$ satisfy $\mb X_0 \mb Y_0^* = \mb M_0$. Note that $(\mb X_0,\mb Y_0)$ would be the \emph{unique} maximizer of \eqref{eq:estimator} if for any $\mb{\varDelta}_1\in\mbb{R}^{d_1\times r}$ and $\mb{\varDelta}_2\in\mbb{R}^{d_2\times r}$ 
\begin{equation}
\label{eq:opt-condition}
\inp{\tilde{\mb X}_0,\mb{\varDelta}_1}+\inp{\tilde{\mb Y}_0,\mb{\varDelta}_2} \leq \frac{1}{n} \sum_{i=1}^{n} \ell_i(\mb X_0 + \mb{\varDelta}_1, \mb Y_0 + \mb{\varDelta}_2) - \ell_i(\mb X_0, \mb Y_0)\,,
\end{equation}
with equality holding only for $\mb \varDelta_1 =\mb 0$ and $\mb \varDelta_2=\mb 0$.  

Since $\mb a_i^* \mb X_0 \mb Y_0^* \mb b_i = m_i$, for each $i$, we obtain 
\begin{align}
& \ell_i(\mb X_0 + \mb{\varDelta}_1, \mb Y_0 + \mb{\varDelta}_2) - \ell_i(\mb X_0, \mb Y_0) \nonumber \\
& = 
\frac{1}{2} \norm{\mb{\varDelta}_1^* \mb a_i}^{2} 
+ \inp{\mb X_0^* \mb a_i,\mb{\varDelta}_1^* \mb a_i}
+ \frac{1}{2} \norm{\mb \varDelta_2^* \mb b_i}^{2} 
+ \inp{\mb Y_0^* \mb b_i,\mb{\varDelta}_2^* \mb b_i} \nonumber \\
& \quad + \left|\mb a_i^* (\mb X_0 \mb{\varDelta}_2^* + \mb{\varDelta}_1 \mb Y_0^* + \mb{\varDelta}_1 \mb{\varDelta}_2^*) \mb b_i\right| \nonumber \\
& \geq 
\inp{\mb a_i \mb a_i^* \mb X_0, \mb{\varDelta}_1}
+ \inp{\mb b_i \mb b_i^* \mb Y_0, \mb{\varDelta}_2}
+ \left|\mb a_i^* (\mb X_0 \mb{\varDelta}_2^* + \mb{\varDelta}_1 \mb Y_0^*) \mb b_i\right| \nonumber 
\\ & \quad - \left|\mb a_i^* \mb{\varDelta}_1 \mb{\varDelta}_2^* \mb b_i\right| 
+ \frac{1}{2} \norm{\mb{\varDelta}_1^* \mb a_i}^{2}
+ \frac{1}{2} \norm{\mb \varDelta_2^* \mb b_i}^{2} \nonumber \\
& \geq \inp{\mb a_i \mb a_i^* \mb X_0, \mb{\varDelta}_1}
+ \inp{\mb b_i \mb b_i^* \mb Y_0, \mb{\varDelta}_2}
+ \left|\mb a_i^* (\mb X_0 \mb{\varDelta}_2^* + \mb{\varDelta}_1 \mb Y_0^*) \mb b_i\right| \,, \label{eq:lb_diff_fidelity}
\end{align}
where the first lower bound is obtained by the triangle inequality and the next lower bound follows from the Cauchy-Schwarz inequality. 

Using \eqref{eq:lb_diff_fidelity}, for $i=1,\dotsc,n$, the right-hand side of \eqref{eq:sufficient condition} can be bounded from above. This bound shows that if \eqref{eq:sufficient condition} holds with equality occurring only at $\left[\mb \varDelta_1;\ \mb \varDelta_2\right]=\mb 0$, then \eqref{eq:opt-condition} holds and the claim is proved.
\end{proof}
For any $\mb X_0 \in \mbb C^{d_1 \times r}$ and $\mb Y_0 \in \mbb C^{d_2 \times r}$ that satisfy $\mb X_0 \mb Y_0^* = \mb M_0$, we have $\mb X_0 \mb \varDelta_2^* + \mb \varDelta_1 \mb Y_0^* \in \mc{T}$. Therefore, by Lemma~\ref{lem:sufficient condition} and \eqref{eq:cond_smallball}, it suffices to show that 
\begin{equation}
\label{eq:suff2}
\inp{\tilde{\mb X}_0 - \frac{1}{n} \mb A \mb A^* \mb{X}_0,\mb \varDelta_1}
+\inp{\conj{\tilde{\mb Y}}_0 - \frac{1}{n} \conj{\mb B} \mb B^\T \conj{\mb{Y}}_0, \conj{\mb\varDelta}_2} 
\le \rho \norm{\mb X_0 \mb \varDelta_2^* + \mb \varDelta_1 \mb Y_0^*}_\F\,
\end{equation}
for all $\mb{\varDelta}_1\in\mbb{C}^{d_1\times r}$ and $\mb{\varDelta}_2\in\mbb{C}^{d_2\times r}$ with the equality only when $[\mb \varDelta_1 ;\ \mb \varDelta_2] = \mb 0$.

Define the linear operator $\mc L: \mbb C^{(d_1+d_2) \times r} \to \mbb C^{d_1 \times d_2}$ by
\[
\mc L\left( \bmx{\mb \varDelta_1 ;\ \conj{\mb \varDelta}_2} \right) = \mb X_0 \mb \varDelta_2^* + \mb \varDelta_1 \mb Y_0^*, \quad \text{for all}\ \mb \varDelta_1 \in \mbb C^{d_1 \times r}\,, \mb \varDelta_2 \in \mbb C^{d_2 \times r}\,, 
\]
whose adjoint operator is \[
\mc L^*(\mb Z) = \bmx{\mb Z \mb Y_0 ;\ \mb Z^\T \conj{\mb X_0}}\,, \quad \text{for all}\  \mb Z \in \mbb C^{d_1 \times d_2}\,.
\]
With 
\[
\mb \varDelta = \bmx{\mb \varDelta_1 ;\ \conj{\mb \varDelta}_2}\,,
\]
and
\begin{align}
    \mb E & = \bmx{ 
    \displaystyle \tilde{\mb X}_{0}-\frac{1}{n} \mb A \mb A^* \mb X_0; \
    \displaystyle \conj{\tilde{\mb Y}}_{0}-\frac{1}{n} \conj{\mb B} \mb B^\T \conj{\mb Y_0}
    }\,,\label{eq:EE}
\end{align}
we can rewrite \eqref{eq:suff2} as 
\[\inp{\mb E, \mb \varDelta} \le \rho \norm{\mc L\left(\mb \varDelta\right)}_\F\,.\]
Note that $\mc L$ generally has a nontrivial nullspace, particularly, if $(d_1+d_2)r < d_1 d_2$. Therefore, in view of the inequality above, it is necessary to have $\inp{\mb E, \mb \varDelta} = 0$ for all $\mb \varDelta$ in the nullspace of $\mc L$. Fortunately, for a certain choice of $(\mb X_0, \mb Y_0)$ the corresponding matrix $\mb E$ satisfies the required condition, as shown by the following lemma, which is proved in Appendix~\ref{sec:proof:lem:E-in-V}.

\begin{lem}
\label{lem:E-in-V}
Let $(\mb X_0, \mb Y_0)$ be the solution to 
\begin{equation}
\label{eq:anchor_opt}
\begin{aligned}
\max_{\mb X\in \mbb C^{d_1\times r}, \mb Y\in \mbb C^{d_2\times r}} 
& \displaystyle \inp{\tilde{\mb X}_{0},\mb X}+\inp{\tilde{\mb Y}_{0},\mb Y}-\frac{1}{2n} \norm{ \mb X^*\mb A}_\F^{2} - \frac{1}{2n} \norm{\mb Y^* \mb B}_\F^{2} 
\\
\mr{subject~to}\ & \mb X \mb Y^* = \mb M_0 \,.
\end{aligned}
\end{equation}
For the operator $\mc L$ and the matrix $\mb E$, defined by \eqref{eq:EE} in terms of the particular solution $(\mb X_0,\mb Y_0)$ above, we have
\[
\mb E \in \mbb V \defeq \mr{range}(\mc L^*) \,.
\]
Furthermore, if \eqref{eq:isometry} holds, then 
\begin{equation}
\label{eq:ubZ0byZ1}
\norm{\bmx{\mb X_0 \\ \mb Y_0} - \bmx{\tilde{\mb X}_0 \\ \tilde{\mb Y}_0}}_\F^2
\leq \frac{1+\eta}{1-\eta} \cdot \norm{\bmx{\mb X \\ \mb Y} - \bmx{\tilde{\mb X}_0 \\ \tilde{\mb Y}_0}}_\F^2
+ \frac{2 \eta}{1-\eta} \cdot \norm{\bmx{\mb X_0 \\ \mb Y_0} - \bmx{\mb X \\ \mb Y}}_\F \norm{\bmx{\tilde{\mb X}_0 \\ \tilde{\mb Y}_0}}_\F
\end{equation}
for all $\bm{X} \in \mbb C^{d_1 \times r}$ and $\bm{Y} \in \mbb C^{d_2 \times r}$ satisfying $\bm{X} \bm{Y}^* = \bm M_0$. 
\end{lem}

Hereafter, the pair $(\mb X_0, \mb Y_0)$ is chosen as in Lemma~\ref{lem:E-in-V}. The subspace $\mbb V$ can be described explicitly as
\begin{equation}
\label{eq:def_spV}
\mbb V = \mr{range}(\mc L^*) = \left\{ \bmx{\mb Z \mb Y_0 ;\ \mb Z^\T \conj{\mb X_0}} : \mb Z \in \mbb C^{d_1 \times d_2} \right\} \,.
\end{equation}
By the fundamental theorem of linear algebra, we also have $\mbb V = \mr{null}(\mc L)^\perp$. Thus, with $\mc P_{\mbb V}$ denoting the orthogonal projection onto the subspace $\mbb V$, Lemma~\ref{lem:E-in-V} implies that $\mb E = \mc P_{\mbb V} \mb E$. Consequently, to guarantee \eqref{eq:suff2}, it suffices to have
\begin{equation}
    \norm{\mb E}_\F\, \norm{\mc P_{\mbb V} \mb \varDelta}_{\F} \le \rho \norm{\mc L(\mc P_{\mbb V} \mb \varDelta)}_\F \,,
\label{eq:HEDelta}
\end{equation}
because by the Cauchy-Schwarz inequality
\begin{align*}
\inp{\mb E, \mb \varDelta} 
= \inp{\mc P_{\mbb V} \mb E, \mb \varDelta}
= \inp{\mb E, \mc P_{\mbb V}\mb \varDelta}
\le \norm{\mb E}_\F\, \norm{\mc P_{\mbb V} \mb \varDelta}_{\F}\,.
\end{align*}
Furthermore, the following technical lemma provides a lower bound for $\norm{\mc L\left(\mc P_{\mbb V}\mb{\varDelta}\right)}_\F/\norm{\mc P_{\mbb V}\mb{\varDelta}}_\F$.

\begin{lem}
The linear operator $\mc L$ satisfies
\begin{align}
\norm{\mc L(P_{\mbb V} \mb \varDelta)}_\F 
&\ge \min\{\sigma_{\min}(\mb X_0), \sigma_{\min}(\mb Y_0)\} \norm{\mc P_{\mbb V} \mb \varDelta}_\F\,, \quad \forall \mb \varDelta \in \mbb C^{(d_1+d_2) \times r}\,. \label{eq:lb_w_V}
\end{align}
\end{lem}

\begin{proof}
Let $[\mb \varDelta_1;\ \conj{\mb \varDelta}_2]$ belong to $\mbb V = \mr{range}(\mc L^*)$. Then there exists $\mb Z \in \mbb C^{d_1 \times d_2}$ such that $\mb \varDelta_1 = \mb Z \mb Y_0$ and $\mb \varDelta_2 = \mb Z^* \mb X_0$. Thus
\begin{align*}
\norm{\mc L([\mb \varDelta_1;\ \conj{\mb \varDelta}_2])}_\F^2 
&= \norm{\mb X_0 \mb \varDelta_2^* + \mb \varDelta_1 \mb Y_0^*}_\F^2 \\
&= \norm{\mb X_0 \mb X_0^* \mb Z + \mb Z \mb Y_0 \mb Y_0^*}_\F^2 \\
&= \norm{\mb X_0 \mb X_0^* \mb Z}_\F^2 
+ \norm{\mb Z \mb Y_0 \mb Y_0^*}_\F^2
+ 2 \langle \mb X_0 \mb X_0^* \mb Z, \mb Z \mb Y_0 \mb Y_0^* \rangle \\
&= \norm{\mb X_0 \mb X_0^* \mb Z}_\F^2 
+ \norm{\mb Z \mb Y_0 \mb Y_0^*}_\F^2
+ 2 \norm{\mb X_0^* \mb Z \mb Y_0}_\F^2 \\
&\ge \sigma_{\min}^2(\mb X_0) \norm{\mb X_0^* \mb Z}_\F^2 + \sigma_{\min}^2(\mb Y_0) \norm{\mb Z \mb Y_0}_\F^2 \\
&\ge \min\{\sigma_{\min}^2(\mb X_0),\sigma_{\min}^2(\mb Y_0)\} \norm{[\mb \varDelta_1;\ \conj{\mb \varDelta}_2]}_\F^2\,.
\end{align*}
\end{proof}

Note that
\begin{align*}
\max_{\mb X\in \mbb C^{d_1\times r}, \mb Y\in \mbb C^{d_2\times r}} &  \min\{\sigma_{\min}(\mb X), \sigma_{\min}(\mb Y)\}  = \sqrt{\sigma_r(\mb M_0)} \\
\mr{subject~to}\ & \mb X \mb Y^* = \mb M_0\,.
\end{align*}
Indeed, the assumptions of the proposition implies that $\min\{\sigma_{\min}(\mb X_0), \sigma_{\min}(\mb Y_0)\}$ is larger than $\sqrt{\sigma_r(\mb M_0)}$ divided by a numerical constant. In order to show this, we introduce another pair $(\mb X_1, \mb Y_1)$ with $\mb X_1 \mb Y_1^* = \mb M_0$ so that $[\mb X_1 ;\ \mb Y_1]$ approximates $[\tilde{\mb X}_0 ;\ \tilde{\mb Y}_0]$. The following lemma provides an upper bound on the approximation error; the proof is provided in Appendix~\ref{sec:proof:lem:sqrtDK}. 

\begin{lem}
\label{lem:sqrtDK}
Suppose that the rank-$r$ matrices $\mb{M}_0$ and $\tilde{\mb{M}}_0$, whose compact SVDs are respectively $\mb{U}_0 \mb{\varSigma}_0 \mb{V}_0^*$ and $\tilde{\mb{U}}_0 \tilde{\mb{\varSigma}}_0 \tilde{\mb{V}}_0^*$, satisfy $\|\tilde{\mb{M}}_0 - \mb{M}_0\| < \sigma_r(\mb{M}_0)$. Then
\begin{align*}
\min_{\mb{Q}\in \mbb{C}^{r\times r} \st \mb Q^{-1} = \mb Q^* } \left\| \begin{bmatrix} \mb{U}_0 \\ \mb{V}_0 \end{bmatrix} \mb{\varSigma}_0^{1/2} \mb{Q} - \begin{bmatrix} \tilde{\mb{U}}_0 \\ \tilde{\mb{V}}_0 \end{bmatrix} \tilde{\mb{\varSigma}}_0^{1/2} \right\|_\F
\leq 
\frac{8\sqrt{2} \, \sigma_1(\mb{\varSigma}_0)}{\sqrt{\sigma_r(\mb{\varSigma}_0)}}
\cdot \frac{\| \tilde{\mb{M}}_0 - \mb{M}_0 \|_\F}{\sigma_r(\mb{\varSigma}_0) - \|\tilde{\mb{M}}_0 - \mb{M}_0\|}.
\end{align*}
\end{lem}

Let $\tilde{\mb U}_0$, $\tilde{\mb \varSigma}_0$ and $\tilde{\mb V}_0$ be as in Lemma~\ref{lem:sqrtDK}. Let $\mb Q$ be the minimizer in Lemma~\ref{lem:sqrtDK}. Let 
\[
\mb X_1 = \mb U_0 \mb \varSigma_0^{1/2} \mb{Q} \quad \text{and} \quad \mb Y_1 = \mb V_0 \mb \varSigma_0^{1/2} \mb{Q} \,.
\]
Then \eqref{eq:cond_init} implies
\begin{equation}
\label{eq:ubZ1minusZ0tilde}
\frac{\norm{[\mb X_1 ;\ \mb Y_1] - [\tilde{\mb X}_0 ;\ \tilde{\mb Y}_0]}_\F}{\sqrt{\sigma_r(\mb M_0)}} 
\leq \frac{8\sqrt{2} \rho}{C_2-\rho/\sqrt{r} \kappa} 
\leq \frac{8\sqrt{2} \rho}{C_2-1} \,.
\end{equation}

Choosing $C_2\ge 8\sqrt{2}+1$ yields 
\begin{equation}
\label{eq:ubZ1minusZ0tilde2}
\norm{[\mb X_1 ;\ \mb Y_1] - [\tilde{\mb X}_0 ;\ \tilde{\mb Y}_0]}_\F 
\leq \rho \sqrt{\sigma_r(\mb M_0)} \,.
\end{equation}
It follows from \eqref{eq:ubZ1minusZ0tilde2} via the triangle inequality that
\begin{equation}
\label{eq:Z0tilde}
\norm{\bmx{\tilde{\mb X}_0 \\ \tilde{\mb Y}_0}}_\F 
\leq \norm{\bmx{\mb X_1 \\ \mb Y_1} - \bmx{\tilde{\mb X}_0 \\ \tilde{\mb Y}_0}}_\F + \norm{\bmx{\mb X_1 \\ \mb Y_1}}_\F 
\leq \rho \sqrt{\sigma_r(\mb M_0)} + \sqrt{r \sigma_1(\mb M_0)} 
\leq 2 \sqrt{r \sigma_1(\mb M_0)} \,.
\end{equation}

Plugging in \eqref{eq:Z0tilde} to \eqref{eq:ubZ0byZ1} with $\mb X = \mb X_1$ and $\mb Y = \mb Y_1$ gives
\begin{align*}
\norm{\bmx{\mb X_0 \\ \mb Y_0} - \bmx{\tilde{\mb X}_0 \\ \tilde{\mb Y}_0}}_\F^2
&\leq \frac{1+\eta}{1-\eta} \cdot \norm{\bmx{\mb X_1 \\ \mb Y_1} - \bmx{\tilde{\mb X}_0 \\ \tilde{\mb Y}_0}}_\F^2 \\
&+ \frac{4 \eta \sqrt{r \sigma_1(\mb M_0)}}{1-\eta} \, 
\left(
\norm{\bmx{\mb X_0 \\ \mb Y_0} - \bmx{\tilde{\mb X}_0 \\ \tilde{\mb Y}_0}}_\F
+ 
\norm{\bmx{\mb X_1 \\ \mb Y_1} - \bmx{\tilde{\mb X}_0 \\ \tilde{\mb Y}_0}}_\F
\right) \,.
\end{align*}
After some simplification, the above inequality and \eqref{eq:ubZ1minusZ0tilde} imply
\begin{align*}
\frac{\norm{[\mb X_0 ;\ \mb Y_0] - [\tilde{\mb X}_0 ;\ \tilde{\mb Y}_0]}_\F}{\rho \sqrt{\sigma_r(\mb M_0)}}
\leq \frac{1}{C_1-1} \left( 2 + \sqrt{\frac{128(C_1+1)}{(C_1-1)(C_2-1)^2}} + \frac{32\sqrt{2}}{(C_1-1)(C_2-1)} + \frac{4}{(C_1-1)^2} \right)
 \,.
\end{align*}

By choosing both $C_1$ and $C_2$ large enough, we obtain
\[
\norm{\bmx{\mb X_1 \\ \mb Y_1} - \bmx{\tilde{\mb X}_0 \\ \tilde{\mb Y}_0}}_\F
\leq \frac{\rho \sqrt{\sigma_r(\mb M_0)}}{10}
\quad \text{and} \quad 
\norm{\bmx{\mb X_0 \\ \mb Y_0} - \bmx{\tilde{\mb X}_0 \\ \tilde{\mb Y}_0}}_\F
\leq \frac{\rho \sqrt{\sigma_r(\mb M_0)}}{10} \,.
\]
Then by the triangle inequality we have
\[
\norm{\bmx{\mb X_0 \\ \mb Y_0} - \bmx{\mb X_1 \\ \mb Y_1}}_\F
\leq \frac{\rho \sqrt{\sigma_r(\mb M_0)}}{5} \,.
\]
By the singular value perturbation theory by Weyl \cite{bhatia2013matrix} and because $\rho \in (0,1)$, it follows that
\[
\sigma_r(\mb X_0) \geq \sigma_r(\mb X_1) - \norm{\mb X_0 - \mb X_1}
\geq \frac{4\sqrt{\sigma_r(\mb M_0)}}{5} 
\]
and
\[
\sigma_r(\mb Y_0) \geq \sigma_r(\mb Y_1) - \norm{\mb Y_0 - \mb Y_1}
\geq \frac{4\sqrt{\sigma_r(\mb M_0)}}{5} \,.
\]

Then \eqref{eq:lb_w_V} is implied by
\begin{equation*}
\norm{\mc L(P_{\mbb V} \mb \varDelta)}_2 
\geq \frac{\sqrt{\sigma_r(\mb M_0)}}{2} \cdot \norm{\mc P_{\mbb V} \mb \varDelta}_\F \,. 
\end{equation*}

Therefore, by combinig the above estimates, we obtain a sufficient condition for \eqref{eq:HEDelta} given by
\begin{equation}
\norm{\mb E}_\F \leq \frac{4 \rho \sqrt{\sigma_r(\mb M_0)}}{5} \,.
\label{eq:ub_E}
\end{equation}

The following lemma, which is proved in Appendix~\ref{sec:proof:lem:upE}, provides an upper bound for $\norm{\mb E}_\F$ that can be used to ensure the sufficient condition \eqref{eq:ub_E}.

\begin{lem}
\label{lem:upE}
Let $(\mb{X}_0, \mb{Y}_0)$ and $\mb E$ be as in Lemma~\ref{lem:E-in-V}. Suppose that \eqref{eq:isometry} holds. Then for all $\left(\mb X,\mb Y\right)$ satisfying $\mb X \mb Y^* = \mb M_0$, we have
\begin{align*}
\norm{\mb E}_\F
\leq
\frac{3+\eta}{1-\eta} \cdot \Bigg(
\norm{[\tilde{\mb{X}}_0 - \mb{X} ;\ \tilde{\mb{Y}}_0 - \mb{Y}]}_\F
+
\eta \norm{[\mb{X};\ \mb{Y}] }_\F
\Bigg) \,.
\end{align*}
\end{lem}

By applying Lemma~\ref{lem:upE} with $\mb X = \mb X_1$ and $\mb Y = \mb Y_1$, and by \eqref{eq:defeta}, we obtain
\[
\norm{\mb E}_\F
\leq
\frac{3C_1+1}{C_1-1} \left(\frac{1}{10} + \frac{1}{C_1} \right) \rho \sqrt{\sigma_r(\mb M_0)} \,.
\]

Finally we choose $C_1$ large enough so that
\[
\frac{3C_1+1}{C_1-1} \left(\frac{1}{10} + \frac{1}{C_1} \right) \leq \frac{4}{5}\,.
\]
This completes the proof.

\section{Analysis of Spectral Initialization}\label{sec:spectral-init}

Note that 
\[
\mc{A}^* \mc{A}(\mb{M}_0) = \frac{1}{n} \sum_{i=1}^n \mb{a}_i \mb{a}_i^* \mb{M}_0 \mb{b}_i \mb{b}_i^* \,.
\]
Then by the independence between $\mb a$ and $\mb b$ together with the isotropy of each of them implies 
\[
\mbb{E} \mc{A}^* \mc{A}(\mb{M}_0) = \mb{M}_0 \,.
\]

Since the spectral initialization computes $\tilde{\mb{M}}_0$ as the best rank-$r$ approximation of $\mc A^* \mc A(\mb{M}_0)$, by the optimality, we have
\[
\|\tilde{\mb{M}}_0 - \mb{M}_0\|
\leq 2 \|(\mc{A}^* \mc{A} - \mr{Id}) \mb{M}_0\| \,.
\]

Our goal is to show that
\begin{equation*}
\|(\mc{A}^* \mc{A} - \mr{Id}) \mb{M}_0\| \leq \frac{C \|\mb{M}_0\|}{\sqrt{r} \kappa^2}
\end{equation*}
for an absolute constant $C$. Then it will imply \eqref{eq:cond_init_intro}. 

Let
\[
\mb Z_i = \frac{1}{n} \mb{a}_i \mb{a}_i^* \mb{M}_0 \mb{b}_i \mb{b}_i^*\,, \quad i=1,\dots,n\,.
\]
Then we will show that 
\begin{equation*}
\Big\|\sum_{i=1}^n \mb Z_i - \E \mb Z_i \Big\| \leq \frac{C\|\mb{M}_0\|}{\sqrt{r}\kappa^2}
\end{equation*}
holds with high probability respectively in the cases of Gaussian and $t$-design measurements.

In the following derivations we use the shorthand
\begin{equation}
\label{eq:defxi}
\xi \defeq \norm{\sum_{i=1}^n \mb Z_i - \E \mb Z_i}_{S_p}\,
\end{equation}
for compact notation.

\subsection{Proof of Proposition~\ref{prop:init_gauss}}

By the triangle inequality we have
\begin{align*}
n \norm{\mb Z_i}_{S_p} 
= \norm{\sum_{k=1}^r \sigma_k \mb a_i \mb a_i^* \mb u_k \mb v_k^* \mb b_i \mb b_i^*}_{S_p} 
\leq \sum_{k=1}^r \sigma_k \norm{\mb a_i}_2 |\mb a_i^* \mb u_k| |\mb v_k^* \mb b_i| \norm{\mb b_i^*}_2\,.
\end{align*}
It follows that
\begin{align*}
\left(\E \norm{\mb Z_i}_{S_p}^p \right)^{1/p}
&\leq \frac{1}{n} \sum_{k=1}^r \sigma_k \left( \E \norm{\mb a_i}_2^p |\mb a_i^* \mb u_k|^p |\mb v_k^* \mb b_i|^p \norm{\mb b_i^*}_2^p \right)^{1/p} \\
&\overset{\mathrm{(a)}}{\leq} \frac{1}{n} \sum_{k=1}^r \sigma_k \left( \E \norm{\mb a_i}_2^p |\mb a_i^* \mb u_k|^p \right)^{1/p} \left( \E |\mb v_k^* \mb b_i|^p \norm{\mb b_i^*}_2^p \right)^{1/p} \\
&\overset{\mathrm{(b)}}{\leq} \frac{1}{n} \sum_{k=1}^r \sigma_k \left( \E \norm{\mb a_i}_2^{2p} \right)^{1/2p} \left( \E |\mb a_i^* \mb u_k|^{2p} \right)^{1/2p} \left( \E |\mb v_k^* \mb b_i|^{2p} \right)^{1/2p} \left( \E \norm{\mb b_i^*}_2^{2p} \right)^{1/2p} \\
&\overset{\mathrm{(c)}}{\lesssim} \frac{p^2 \sqrt{d_1 d_2}}{n} \sum_{k=1}^r \sigma_k \\
&\lesssim \frac{\norm{\mb M_0} p^2 r (d_1+d_2)}{n}\,,
\end{align*}
where (a) holds by the independence between $\mb a_i$ and $\mb b_i$; (b) follows from Cauchy-Schwarz inequality; and (c) holds by \cite[Eq. (2.15) and Lemma 2.7.6]{vershynin2018high} since $\norm{\mb a_i}_2^2$, $|\mb a_i^* \mb u_k|^2$, $|\mb v_k^* \mb b_i|^2$, and $\norm{\mb b_i}_2^2$ are sub-exponential random variables. 
Then we deduce that
\begin{align*}
\left( \E \norm{\mb Z_i - \E \mb Z_i}_{S_p}^p \right)^{1/p} 
\leq 2 \left( \E \norm{\mb Z_i}_{S_p}^p \right)^{1/p}
\lesssim \frac{\norm{\mb M_0} p^2 r (d_1+d_2)}{n}\,.
\end{align*}

Furthermore, the second-order moments are computed as
\[
\mbb{E} (\mb Z_i - \E \mb Z_i) (\mb Z_i - \E \mb Z_i)^* = \frac{d_2+2}{n^2} \|\mb M_0\|_\F^2 \mb I_{d_1} + \frac{2d_2+3}{n^2} \mb M_0 \mb M_0^*
\]
and
\[
\mbb{E} (\mb Z_i - \E \mb Z_i)^* (\mb Z_i - \E \mb Z_i) = \frac{d_1+2}{n^2} \|\mb M_0\|_\F^2 \mb I_{d_2} + \frac{2d_1+3}{n^2} \mb M_0^* \mb M_0\,. 
\]
Then we obtain
\begin{align*}
\left(
\norm{\sum_{i=1}^n \mbb{E} (\mb Z_i - \E \mb Z_i) (\mb Z_i - \E \mb Z_i)^*}_{S_p}
\vee
\norm{\sum_{i=1}^n \mbb{E} (\mb Z_i - \E \mb Z_i)^* (\mb Z_i - \E \mb Z_i)}_{S_p}
\right)
\lesssim \frac{\norm{\mb M_0}^2 r (d_1+d_2)^{1+1/p}}{n} \,.
\end{align*}

Now we are ready to apply Theorem~\ref{thm:matrosen}. Recalling \eqref{eq:defxi}, Theorem~\ref{thm:matrosen} provides 
\begin{align*}
\frac{\left(\E \xi^p\right)^{1/p}}{\norm{\mb M_0}}
\lesssim 
\left( \frac{p r (d_1+d_2)^{1+1/p}}{n} \right)^{1/2}
\vee
\frac{p^3 n^{1/p} r (d_1+d_2)}{n}\,.
\end{align*}
By Markov's inequality, it implies
\begin{equation}
\label{eq:markovconsq}
\P\{\xi \geq (\E \xi^p)^{1/p} \eta^{-1/p} \} \leq \eta.
\end{equation}
We choose $\eta = (d_1+d_2)^{-\alpha}$. 
Then there exists an absolute constant $C_1$ such that if 
\begin{equation}
\label{eq:lb_n}
n \geq C_1 \left[ \epsilon^{-2} p r d^{1+1/p+2\alpha/p} 
\vee 
\left(
\epsilon^{-1} p^3 r d^{1+\alpha/p}
\right)^{p/(p-1)}
\right]\,,
\end{equation}
then 
\begin{equation}
\label{eq:tail_Sq}
\P\{\xi \geq \epsilon \norm{\mb M_0}\} \leq (d_1+d_2)^{-\alpha}\,.
\end{equation}
Since the Shatten-$p$ norm is larger than the spectral norm, it follows that \eqref{eq:tail_Sq} implies 
\[
\P\left\{ \left\| \mc A^* \mc A \mb M_0 - \E \mc A^* \mc A \mb M_0 \right\| \geq \epsilon \norm{\mb M_0} \right\} \leq (d_1+d_2)^{-\alpha}\,.
\]
We choose $p = (2\alpha+1) \ln (d_1+d_2)$ so that the condition in \eqref{eq:lb_n} simplifies to 
\[
n \geq C_2 r(d_1+d_2) \left( \epsilon^{-2} \alpha \ln^2 (d_1+d_2) 
\vee 
\epsilon^{-1} \alpha^3 \ln^3 (d_1+d_2)
\right)
\]
for another an absolute constant $C_2$. 
We obtain a sufficient condition by choosing $\epsilon = C/\kappa^2\sqrt{r}$. 
This compltes the proof.

\subsection{Proof of Proposition~\ref{prop:init_2design}}

Since $\norm{\mb a_i}_2 = \sqrt{d_1}$ and $\norm{\mb b_i}_2 = \sqrt{d_2}$, we have
\[
\norm{\mb Z_i}_{S_{2t}} = \frac{\sqrt{d_1 d_2}}{n} \left|\sum_{k=1}^r \sigma_k \mb a_i^* \mb u_k \mb v_k^* \mb b_i\right|\,.
\]
Note that we have
\begin{align*}
\left|\sum_{k=1}^r \sigma_k \mb a_i^* \mb u_k \mb v_k^* \mb b_i\right|
\leq 
\sum_{k=1}^r \sigma_k \left|\mb a_i^* \mb u_k \mb v_k^* \mb b_i\right|\,.
\end{align*}
Thus, by the triangle inequality in $L_{2t}$ together with the homogeneity of the $L_{2p}$ norm, we obtain
\begin{equation}
\label{eq:calc_spnorm_Yi}
\begin{aligned}
\left( \E \norm{\mb Z_i}_{S_{2t}}^{2t} \right)^{1/2t} 
&\leq \frac{\sqrt{d_1d_2}}{n} \sum_{k=1}^r \sigma_k \left(\E \left|\mb a_i^* \mb u_k \mb v_k^* \mb b_i \right|^{2t} \right)^{1/2t} \\
&\leq \frac{d_1d_2}{n} \sum_{k=1}^r \sigma_k \left(\E \left|d_1^{-1/2} \mb a_i^* \mb u_k\right|^{2t} \right)^{1/2t} \left(\E \left|d_2^{-1/2} \mb b_i^* \mb v_k\right|^{2t} \right)^{1/2t}\,,
\end{aligned}
\end{equation}
where the second inequality holds since $\mb a_i$ and $\mb b_i$ are independent.

Recall that, by the construction of a $t$-design set, $\mb a_i$ satisfies
\[
\E \left(d_1^{-1} \mb a_i \mb a_i^*\right)^{\otimes t} = {\binom{d+t-1}{t}}^{-1} \mc P_{\mathrm{Sym}^t}\,.
\]
Therefore we have
\begin{align*}
\E \left|d_1^{-1/2} \mb a_i^* \mb u_k\right|^{2t} 
&= \mathrm{tr}\left( \E \left(d_1^{-1} \mb a_i \mb a_i^*\right)^{\otimes t} (\mb u_k \mb u_k^*)^{\otimes t} \right) \\
&= {\binom{d_1+t-1}{t}}^{-1} \mathrm{tr} \left( \mc P_{\mathrm{Sym}^t} (\mb u_k \mb u_k^*)^{\otimes t} \right) \\
&= {\binom{d_1+t-1}{t}}^{-1} \mathrm{tr} \left( (\mb u_k \mb u_k^*)^{\otimes t} \right) \leq d_1^{-t} t!\,,
\end{align*}
where the third identity follows since $(\mb u_k \mb u_k^*)^{\otimes t}$ is invariant under the factor-wise permutation. 
Then we obtain 
\begin{equation}
\label{eq:calc_2p_a}
\left( \E \left|d_1^{-1/2} \mb a_i^* \mb u_k\right|^{2t} \right)^{1/2t} 
\leq \left( d_1^{-t} t! \right)^{1/2t} \leq \sqrt{\frac{t}{d_1}}\,.
\end{equation}
Similarly we also have
\begin{equation}
\label{eq:calc_2p_b}
\left( \E \left|d_2^{-1/2} \mb b_i^* \mb v_k\right|^{2t} \right)^{1/2t} 
\leq \left( d_2^{-t} t! \right)^{1/2t} \leq \sqrt{\frac{t}{d_2}}\,.
\end{equation}
Then by plugging in the upper bounds in \eqref{eq:calc_2p_a} and \eqref{eq:calc_2p_b} into \eqref{eq:calc_spnorm_Yi}, we obtain
\[
\left( \E \norm{\mb Z_i}^{2t} \right)^{1/2t} 
\leq \frac{t \sigma_1 r \sqrt{d_1 d_2}}{n}
\leq \frac{t \sigma_1 r (d_1+d_2)}{2n}\,.
\]
Furthermore, the triangle inequality and Jensen's inequality yield
\begin{align*}
\left( \E \norm{\mb Z_i - \E \mb Z_i}_{S_{2t}}^{2t} \right)^{1/2t} 
& \leq \left( \E \norm{\mb Z_i}_{S_{2t}}^{2t} \right)^{1/2t} 
+ \left( \E \norm{\E \mb Z_i}_{S_{2t}}^{2t} \right)^{1/2t} \\
& = \left( \E \norm{\mb Z_i}_{S_{2t}}^{2t} \right)^{1/2t} 
+ \norm{\E \mb Z_i}_{S_{2t}} \\ 
& \leq 2 \left( \E \norm{\mb Z_i}_{S_{2t}}^{2t} \right)^{1/2t}\,.
\end{align*}
Putting these bounds together we obtain
\[
\left( \E \norm{\mb Z_i}^{2t} \right)^{1/2t} 
\leq \frac{2 t \sigma_1 r (d_1+d_2)}{2n}\,.
\]

Furthermore, the second moments are computed as
\[
\mbb{E} (\mb Z_i - \E \mb Z_i) (\mb Z_i - \E \mb Z_i)^* 
= \frac{1}{n^2} \left( \frac{d_1d_2 \|\mb{M}_0\|_\F^2 \mb{I}_{d_1}}{d_1+1} + \left(\frac{d_1d_2}{d_1+1} - 1 \right) \mb{M}_0 \mb{M}_0^* \right)
\]
and
\[
\mbb{E} (\mb Z_i - \E \mb Z_i)^* (\mb Z_i - \E \mb Z_i)
= \frac{1}{n^2} \left( \frac{d_1d_2 \|\mb{M}_0\|_\F^2 \mb{I}_{d_2}}{d_2+1} + \left(\frac{d_1d_2}{d_2+1} - 1 \right) \mb{M}_0^* \mb{M}_0 \right)\,. 
\]
Then we have
\begin{align*}
\left(
\norm{\sum_{i=1}^n \mbb{E} (\mb Z_i - \E \mb Z_i) (\mb Z_i - \E \mb Z_i)^*}_{S_{2t}}
\vee
\norm{\sum_{i=1}^n \mbb{E} (\mb Z_i - \E \mb Z_i)^* (\mb Z_i - \E \mb Z_i)}_{S_{2t}}
\right)
\leq \frac{4 \sigma_1^2 r (d_1+d_2)^{1+1/2t}}{n} \,.
\end{align*}

With $\xi$ defined in \eqref{eq:defxi}, Theorem~\ref{thm:matrosen} provides 
\begin{align*}
\left(\E \xi^{2t}\right)^{1/2t} 
\lesssim 
\left( \frac{t \sigma_1^2 r (d_1+d_2)^{1+1/2t}}{n} \right)^{1/2}
\vee
\frac{t^2 n^{1/2t} \sigma_1 r (d_1+d_2)}{n}\,.
\end{align*}
Then invoking \eqref{eq:markovconsq} with $\eta = (d_1+d_2)^{-\alpha}$, we can show that there exists an absolute constant $C_1$ such that if 
\[
n \geq C_1 \left[ \epsilon^{-2} t r (d_1+d_2)^{1+1/2t+\alpha/t} 
\vee 
\left(
\epsilon^{-1} t^2 r (d_1+d_2)^{1+\alpha/2t}
\right)^{2t/(2t-1)}
\right]\,,
\]
then 
\[
\P\left\{ \left\| \mc A^* \mc A \mb M_0 - \E \mc A^* \mc A \mb M_0 \right\| \geq \epsilon \norm{\mb M_0} \right\} \leq d^{-\alpha}\,.
\]
The desired result follows by choosing $\epsilon = C/\kappa^2\sqrt{r}$.

\appendix

\section{Tools from Random Matrix Theory}

\begin{theorem}[{\cite[Theorem~II.13]{davidson2001local}}]
\label{thm:wishart}
Let $\mb G \in \mbb R^{m \times n}$ be a random matrix whose entries are independent copies of $g \sim \mc N(0,1)$. Then 
\[
\mbb{P}\left\{
\left\|
\frac{1}{m} \mb G^\T \mb G - \mb I_n
\right\|
> 3 \max\left( \sqrt{\frac{4n}{m}}, \frac{4n}{m} \right)
\right\}
\leq 2 \exp(-n/2) \,.
\]
\end{theorem}

\begin{theorem}[{Matrix Bernstein inequality \cite[Theorem~1.6]{tropp2012user}}]
\label{thm:matberntropp}
Let $(\mb Y_k) \subset \mbb{C}^{m \times n}$ be a finite sequence of independent zero-mean random matrices such that $\|\mb Y_k\| \leq R$ almost surely for all $k$. Let
\begin{equation*}
\sigma^2 = \max\left\{ \left\| \sum_k \mbb{E} \mb Y_k \mb Y_k^* \right\|, \, \left\| \sum_k \mbb{E} \mb Y_k^* \mb Y_k \right\| \right\}\,.
\end{equation*}
Then for all $t > 0$
\[
\mbb{P} \left\{
\left\| \sum_k \mb Y_k \right\| \geq t
    \right\} 
\leq (m+n) \cdot \exp\left( \frac{-t^2/2}{\sigma^2 + Rt/3} \right) \,.
\]
\end{theorem}

\begin{theorem}[{Noncommutative Rosenthal inequality \cite[Theorem~0.4]{junge2013noncommutative}, \cite[Theorem~3.8]{dirksen2011}}]
\label{thm:matrosen}
Let $(\mb Y_k) \subset \mbb C^{d_1 \times d_2}$ be a finite sequence of independent zero-mean random matrices. 
Then there exists an absolute constant $C>0$ such that for all $2 \leq p < \infty$
\begin{align*}
& \left(\E \, \Big\| \sum_k \mb Y_k \Big\|_{S_p}^p\right)^{1/p}\\
& \leq C \Bigg[
\sqrt{p} \Bigg( \Bigg\| \sum_k \mbb{E} \mb Y_k \mb Y_k^* \Bigg\|_{S_p} \vee \Bigg\| \sum_k \mbb{E} \mb Y_k^* \mb Y_k \Bigg\|_{S_p} \Bigg)^{1/2} 
\vee
p \Bigg( \E \sum_k \norm{\mb Y_k}_{S_p}^{p} \Bigg)^{1/p}
\Bigg] \,.
\end{align*}
\end{theorem}

\section{Proofs of the main lemmas}
\subsection{Proof of Lemma~\ref{lem:lbprob}}
\label{sec:proof:lem:lbprob}
Let $\tau > 0$ be fixed. Since $\mb a$ and $\mb b$ are independent and isotropic, they satisfy
\begin{equation}
\label{eq:isotropy-assumption}
  \mbb{E} \bm{a}\mb{a}^* \otimes \bm{b}\bm{b}^* = \mb{I}_{d_1 d_2} \,,
\end{equation}
which implies
\begin{equation}
\label{eq:2ndmoment}
\E \left|\mb{a}^*\mb{H}\mb{b}\right|^2
= \norm{\mb H}_\F^2, \quad \forall \mb{H} \in \mbb R^{d_1 \times d_2} \,.
\end{equation}

Therefore, the Paley-Zygmund inequality \cite{paley1932note} (see also \cite[Corollary 3.3.2]{de2012decoupling}), we have
\begin{equation}
\label{eq:P_i-lowerbound}
\begin{aligned}
    \P\left(\left|\mb a^* \mb H \mb b\right| \ge \tau \norm{\mb H}_\F \right) 
    &= \P\left(\left|\mb a^* \mb H \mb b\right|^2 \ge \tau^2 \E  \left|\mb a^* \mb H \mb b\right|^2 \right) \\
    & \ge \frac{(1 - \tau^2)^2\left(\E \left|\mb a^* \mb H \mb b\right|^2\right)^2}{ \E {\left|\mb a^* \mb H \mb b\right|}^4 }\,. 
\end{aligned}
\end{equation}

Then it suffices to show that the fourth order moment $\E {\left|\mb{a}^*\mb{H}\mb{b}\right|}^4$ is upper-bounded by $\left(\E \left|\mb a^* \mb H \mb b\right|^2\right)^2 = \norm{\mb H}_\F^4$ within a constant factor. We first show this for the real Gaussian case. Since $[\mb{a};\ \mb{b}] \sim \mc{N}(\mb{0},\mb{I}_{d_1+d_2})$, we have
\begin{align}
\E {\left|\mb{a}^*\mb{H}\mb{b}\right|}^4
&= \E_{\mb{b}} (\mb{b}^*\mb{H}^* \otimes \mb{b}^*\mb{H}^*) 
\E_{\mb{a}} (\mb{a}\mb{a}^* \otimes \mb{a}\mb{a}^*)
(\mb{H}\mb{b} \otimes \mb{H}\mb{b}) \nonumber \\
&= \E_{\mb{b}} \norm{\mb{b}^*\mb{H}^*}_{\F}^4 + 2 \, \norm{\mb{b}^*\mb{H}^*}_{S_4}^4 \nonumber \\
&= 3 \, \E_{\mb{b}} \norm{\mb{b}^*\mb{H}^*}_2^4 \nonumber \\
&= 3 \, \tr \left[ (\mb{H} \otimes \mb{H}) 
\E_{\mb{b}} (\mb{b}\mb{b}^* \otimes \mb{b}\mb{b}^*)
(\mb{H}^* \otimes \mb{H}^*) \right] \nonumber \\
&= 3 \norm{\mb{H}}_{\F}^4 + 6 \, \norm{\mb{H}}_{S_4}^4 \nonumber \\
&\leq 9 \norm{\mb{H}}_{\F}^4\,, \label{eq:4thmoment}
\end{align}
where $\norm{\cdot}_{S_4}$ denotes the Schatten-$4$ norm. 

By plugging in \eqref{eq:2ndmoment} and \eqref{eq:4thmoment} to \eqref{eq:P_i-lowerbound}, we obtain
\[
\P\left(\left|\mb{a}^* \mb{H} \mb{b}\right| \ge \tau \norm{\mb{H}}_\F \right) 
\geq \frac{(1-\tau^2)^2}{18} \,.
\]

We obtain an analogous upper bound in the $t$-design case. With the isotropic normalization, $\mb{a}$ and $\mb{b}$ satisfy \eqref{eq:4m_2design_a} and \eqref{eq:4m_2design_b}. Therefore,
\begin{align}
\E {\left|\mb{a}^*\mb{H}\mb{b}\right|}^4
&= \E_{\mb{b}} (\mb{b}^*\mb{H}^* \otimes \mb{b}^*\mb{H}^*) 
\E_{\mb{a}} (\mb{a}\mb{a}^* \otimes \mb{a}\mb{a}^*)
(\mb{H}\mb{b} \otimes \mb{H}\mb{b}) \nonumber \\
&= \frac{d_1}{d_1+1} \left( \E_{\mb{b}} \norm{\mb{b}^*\mb{H}^*}_{\F}^4 + \norm{\mb{b}^*\mb{H}^*}_{S_4}^4 \right) \nonumber \\
&= \frac{2d_1}{d_1+1} \cdot \E_{\mb{b}} \norm{\mb{b}^*\mb{H}^*}_2^4 \nonumber \\
&= \frac{2d_1}{d_1+1} \cdot \tr \left[ (\mb{H} \otimes \mb{H}) 
\E_{\mb{b}} (\mb{b}\mb{b}^* \otimes \mb{b}\mb{b}^*)
(\mb{H}^* \otimes \mb{H}^*) \right] \nonumber \\
&= \frac{2d_1 d_2}{(d_1+1)(d_2+1)} \cdot \left(\norm{\mb{H}}_{\F}^4 + \norm{\mb{H}}_{S_4}^4 \right) \nonumber \\
&\leq \frac{4d_1 d_2}{(d_1+1)(d_2+1)} \cdot \norm{\mb{H}}_{\F}^4 \,. \label{eq:4thmoment_tdesign}
\end{align}
By plugging in \eqref{eq:2ndmoment} and \eqref{eq:4thmoment_tdesign} to \eqref{eq:P_i-lowerbound}, we obtain
\[
\P\left(\left|\mb{a}^* \mb{H} \mb{b}\right| \ge \tau \norm{\mb{H}}_\F \right) 
\geq \frac{(1-\tau^2)^2}{8} \,.
\]

\subsection{Proof of Lemma~\ref{lem:ubradcpl}}
\label{sec:proof:lem:ubradcpl}

The independence and isotropy assumptions imply \eqref{eq:isotropy-assumption}. Note that any $\mb{H} \in \mc T$ can be written as $\mb{H} = \mb{\varDelta}_{1} \mb V_0^* + \mb U_0 \mb{\varDelta}_{2}^*$. Furthermore, without loss of generality, we may assume 
\begin{equation}
\label{eq:decomp_H}
\langle \mb{\varDelta}_{1}, \mb U_0 \rangle = 0 \,,
\end{equation}
which implies
\[
\norm{\mb{H}}_\F^2 = \norm{\mb{\varDelta}_1}_\F^2 + \norm{\mb{\varDelta}_2}_\F^2 \,.
\]

Since
\begin{align*}
\mb a_i^* \mb{H} \mb b_i
& = \tr\left(\left(\mb a_i \mb b_i^* \mb V_0\right)^* \mb{\varDelta}_{1} \right) 
+ \tr\left(\left(\mb U_0^* \mb a_i \mb b_i^*\right)^* \mb{\varDelta}_{2}^* \right) \,,
\end{align*}
by \eqref{eq:decomp_H} we have
\begin{align*}
\mfk{C}_n(\mc T) 
& \le \E \sup_{\mb{H} \in \mc T \cap \mbb{S}} \left| \frac{1}{\sqrt{n}} \sum_{i=1}^n \varepsilon_i \mb a_i^* \mb{H} \mb b_i \right| \\
& \leq \E \sup_{\norm{\left[\mb{\varDelta}_1;\ \mb{\varDelta}_2\right]}_\F = 1} \norm{\mb{\varDelta}_1}_\F \cdot
\left\| \frac{1}{\sqrt{n}} \sum_{i=1}^n \varepsilon_i \mb a_i \mb b_i^* \mb V_0 \right\|_\F
+ \norm{\mb{\varDelta}_2}_\F \cdot
\left\| \frac{1}{\sqrt{n}} \sum_{i=1}^n \varepsilon_i \mb U_0^* \mb a_i \mb b_i^* \right\|_\F \\
& = \E \left( \left\| \frac{1}{\sqrt{n}} \sum_{i=1}^n \varepsilon_i \mb a_i \mb b_i^* \mb V_0 \right\|_\F^2
+ \left\| \frac{1}{\sqrt{n}} \sum_{i=1}^n \varepsilon_i \mb U_0^* \mb a_i \mb b_i^* \right\|_\F^2 \right)^{1/2} \\
& \leq \left( \E  \left\| \frac{1}{\sqrt{n}} \sum_{i=1}^n \varepsilon_i \mb a_i \mb b_i^* \mb V_0 \right\|_\F^2
+ \E \left\| \frac{1}{\sqrt{n}} \sum_{i=1}^n \varepsilon_i \mb U_0^* \mb a_i \mb b_i^* \right\|_\F^2 \right)^{1/2} \\
& = \sqrt{(d_1 + d_2) r} \,,
\end{align*}
where the second step follows by the Cauchy-Schwarz inequality, the fourth step is obtained by Jensen's inequality, and the last step holds since $(\varepsilon_i)_{i=1}^n$ is a Rademacher sequence and $(\mb{a},\mb{b})$ satisfies \eqref{eq:isotropy-assumption}.  

\subsection{Proof of Lemma~\ref{lem:E-in-V}}
\label{sec:proof:lem:E-in-V}
This lemma basically follows from requiring stationarity at $\mb Q=\mb I$ which is away from singularities of the objective. However, to avoid complications arising from derivatives with respect to complex-valued variables, we provide a ``lower level'' proof. To show this claim, first observe that any $(\mb X,\mb Y)$ satisfying $\mb X \mb Y^* = \mb M_0$ can be parameterized by an \emph{invertible} $r\times r$ matrix $\mb Q$ as $(\mb{X},\mb{Y}) = (\mb{X}_0\mb{Q},\mb{Y}_0\mb{Q}^{-*})$. Let $(\mb X_0,\mb Y_0)$ be the maximizer considered in the statement of the lemma. 
Therefore, for any invertible $\mb Q$ in $\mr{GL}_r(\mbb C)$, the set of invertible $r\times r$ complex matrices, we should necessarily have 
\begin{align*}
     & \inp{\tilde{\mb X}_0,\mb X_0\mb Q}+ \inp{\tilde{\mb Y}_0,\mb Y_0{\mb Q}^{-*}}-\frac{1}{2n}\norm{\mb Q^*\mb X_0^*\mb A}_\F^2-\frac{1}{2n}\norm{\mb Q^{-1}\mb Y_0^*\mb B}_\F^2 \\
     & \le \inp{\tilde{\mb X}_0,\mb X_0}+ \inp{\tilde{\mb Y}_0,\mb Y_0}-\frac{1}{2n}\norm{\mb X_0^*\mb A}_\F^2-\frac{1}{2n}\norm{\mb Y_0^*\mb B}_\F^2\,,
\end{align*}
which is equivalent to 
\begin{align*}
    & \inp{\tilde{\mb X}_0,\mb X_0\mb Q-\mb X_0}+ \inp{\tilde{\mb Y}_0,\mb Y_0{\mb Q}^{-*}-\mb Y_0} -\frac{1}{n}\inp{\left(\mb Q-\mb I_r\right)^*\mb X_0^*\mb A, \mb X_0^*\mb A}\\
    &-\frac{1}{n}\inp{(\mb Q^{-1}-\mb I_r)\mb Y_0^*\mb B, \mb Y_0^*\mb B}
     -\frac{1}{2n}\norm{\left(\mb Q-\mb I_r\right)^*\mb X_0^*\mb A}_\F^2-\frac{1}{2n}\norm{(\mb Q^{-1}-\mb I_r)\mb Y_0^*\mb B}_\F^2 \le 0\,.
\end{align*}
To simplify the notation let us define the short-hands 
\[
\mb \varTheta_X = \mb X_0^*\left(\tilde{\mb X}_{0}-\frac{1}{n}\mb A\mb A^*\mb X_0\right)\,,
\]
and 
\[
\mb\varTheta_Y =  \mb Y_0^*\left(\tilde{\mb Y}_{0} - \frac{1}{n}\mb{B}\mb{B}^*\mb Y_0\right) \,.
\] 
The necessary inequality can be expressed as
\begin{align}
\begin{aligned}
     & \inp{\mb \varTheta_X,\mb Q -\mb I_r} + \inp{\mb \varTheta_Y,(\mb Q^{-1}-\mb I_r)^*}\\
     & \le \frac{1}{2n}\norm{(\mb Q-\mb I_r)^*\mb X_0^*\mb A}_\F^2+ \frac{1}{2n}\norm{(\mb Q^{-1}-\mb I_r)\mb Y_0^*\mb B}_\F^2\,.
\end{aligned}
 \label{eq:local-max}
\end{align}
Choosing $\mb Q$ within an arbitrarily small neighborhood of $\mb I_r$ allows us to use the identity 
\[
\mb Q^{-1} -\mb I_r = \sum_{k=1}^\infty (-1)^k(\mb Q-\mb I_r)^k. 
\]
Applying this identity in \eqref{eq:local-max} yields
\begin{align*}
    & \inp{\mb \varTheta_X - \mb \varTheta_Y^*, \mb Q - \mb I_r} + \sum_{k=2}^\infty (-1)^k\inp{\mb \varTheta_Y^*,\left(\mb Q-\mb I_r\right)^k}\\
    & \le \frac{1}{2n}\norm{(\mb Q-\mb I_r)^*\mb X_0^*\mb A}_\F^2+ \frac{1}{2n}\norm{(\mb Q^{-1}-\mb I_r)\mb Y_0^*\mb B}_\F^2\\
    & \le \frac{1}{2n}\norm{\mb Q - \mb I_r}^2\norm{\mb X_0^*\mb A}_\F^2 + \frac{1}{2n}\norm{\mb Q - \mb I_r}^2\norm{\mb Q^{-1}}^2\norm{\mb Y_0^*\mb B}_\F^2\,.
\end{align*}
As $\mb Q\to \mb I_r$, the terms linear in $\mb Q-\mb I_r$ dominate and the terms with superlinear dependence on $\norm{\mb Q-\mb I_r}$ vanish faster. Therefore, the inequality above implies 
\begin{equation}
\label{eq:varThetaXY}
\mb \varTheta_X=\mb \varTheta_Y^* \,.
\end{equation}

Let us express \eqref{eq:EE} as $\mb E = [\mb E_X; \mb E_Y]$ with $\mb E_X = \tilde{\mb X}_{0}-\frac{1}{n}\mb A\mb A^*\mb X_0$ and $\mb E_Y =  \conj{\tilde{\mb Y}_{0}} - \frac{1}{n} \conj{\mb{B}} \mb{B}^\T \conj{\mb Y_0}$. Then \eqref{eq:varThetaXY} implies that $[\mb E_X;\ \mb E_Y]$ belongs to the subspace 
\[
\mbb W \defeq \left\{ \bmx{\mb E_1;\ \mb E_2} \,:\, \mb E_1 \in \mbb C^{d_1 \times r}, \mb E_2 \in \mbb C^{d_2 \times r}, \mb X_0^* \mb E_1 = \mb E_2^\T \mb Y_0 \right\} \,.
\]
It only remains to show that $\mbb W$ coincides with $\mbb V = \mr{null}(\mc L)^\perp$ defined in \eqref{eq:def_spV}. 

We first verify that $\mbb V \subseteq \mbb W$. By the definition of $\mc T$, we can express $\mbb V$ equivalently as
\begin{equation*}
\mbb V = \left\{ \bmx{\mb Z \mb Y_0 ;\ \mb Z^\T \conj{\mb X_0}\,} : \mb Z \in \mc T \right\} \,.
\end{equation*}Suppose that $[\mb E_1 ;\ \mb E_2]$ belongs to $\mbb V$. Then there exists $\mb Z \in \mc T$ such that $\mb E_1 = \mb Z \mb Y_0$ and $\mb E_2 = \mb Z^\T \conj{\mb X_0}$. Therefore,
\[
\mb X_0^* \mb E_1 = \mb X_0^* \mb Z \mb Y_0 
= (\mb Z^\T \conj{\mb X_0})^\T \mb Y_0 = \mb E_2^\T \mb Y_0 \,,
\]
which implies $\bmx{\mb E_1;\ \mb E_2}$ is contained in $\mbb W$. Since $[\mb E_1 ;\ \mb E_2]$ can be chosen arbitrarily in $\mbb V$, we have shown $\mbb V \subseteq \mbb W$. 
    
Furthermore, the orthogonal complement of $\mbb W$ within $\mbb C^{(d_1+d_2)\times r}$ can be  written as 
\[\mbb W^\perp = \left\{\bmx{\mb X_0 \mb S; -\conj{\mb Y_0} \mb S^\T}\st \mb S\in \mbb C^{r\times r}\right\}\,.\] 
Since both $\mb X_0$ and $\mb Y_0$ are full column rank (otherwise, the rank of $\mb M_0=\mb X_0 \mb Y_0^*$ would be smaller than $r$), we deduce that the dimension of $\mbb W^\perp$ is $r^2$, thereby the dimension of $\mbb W$ is $r(d_1+d_2-r)$. In view of the inclusion $\mbb V\subseteq \mbb W$, it only remains to show that the dimension of $\mbb V$ is $r(d_1+d_2-r)$. We do so by arguing that the linear function $\mb Z \mapsto [\mb Z\mb Y_0;\ \mb Z^\T \conj{\mb X_0}]$ is a \emph{bijection} from $\mc T$ to $\mbb V$, or equivalently if $\mb Z\in \mc T$ is mapped to $\mb 0 \in \mbb V$, then $\mb Z = \mb 0$. If for some $\mb Z= \mb\varDelta_1 \mb Y_0^* + \mb X_0 \mb\varDelta_2^* \in \mc T$ we have $\mb X_0^* \mb Z = \mb 0$ and $\mb Z \mb Y_0=\mb 0$, then we should have $\norm{\mb Z}_\F^2 = \tr(\mb Z^* \mb Z) = 0$, and consequently $\mb Z = \mb 0$. Therefore, we have shown  that $\mbb V = \mbb W$, and particularly
\[
\mb E = [\mb E_X ;\ \mb E_Y] \in \mbb W = \mbb V \,.
\]

Next we prove the second part of the lemma. By \eqref{eq:isometry}, we have
\begin{align*}
& (1-\eta) \, \norm{\mb X_0 - \tilde{\mb X}_0}_\F^2 
+ (1-\eta) \, \norm{\mb Y_0 - \tilde{\mb Y}_0}_\F^2 \\
& \leq \frac{1}{n} \, \norm{\mb A^* (\mb X_0 - \tilde{\mb X}_0)}_\F^2
+ \frac{1}{n} \, \norm{\mb B^* (\mb Y_0 - \tilde{\mb Y}_0)}_\F^2 \\
& = \frac{1}{n} \, \norm{\mb A^* \mb X_0}_\F^2
- 2 \left\langle \mb X_0, \frac{1}{n} \mb A \mb A^* \tilde{\mb X}_0\right\rangle
+ \frac{1}{n} \, \norm{\mb A^* \tilde{\mb X}_0}_\F^2 \\
& \quad + \frac{1}{n} \, \norm{\mb B^* \mb Y_0}_\F^2
- 2 \left\langle \mb Y_0, \frac{1}{n} \mb B \mb B^* \tilde{\mb Y}_0 \right\rangle
+ \frac{1}{n} \, \norm{\mb B^* \tilde{\mb Y}_0}_\F^2 \\
& = \frac{1}{n} \, \norm{\mb A^* \mb X_0}_\F^2
- 2 \inp{\mb X_0, \tilde{\mb X}_0}
+ 2 \left\langle \mb X_0, \left(\mb I_{d_1} - \frac{1}{n} \mb A \mb A^*\right) \tilde{\mb X}_0 \right\rangle
+ \frac{1}{n} \, \norm{\mb A^* \tilde{\mb X}_0}_\F^2 \\
& \quad + \frac{1}{n} \, \norm{\mb B^* \mb Y_0}_\F^2
- 2 \inp{\mb Y_0, \tilde{\mb Y}_0}
+ 2 \left\langle \mb Y_0, \left(\mb I_{d_2} - \frac{1}{n} \mb B \mb B^*\right) \tilde{\mb Y}_0 \right\rangle
+ \frac{1}{n} \, \norm{\mb B^* \tilde{\mb Y}_0}_\F^2 
\,.
\end{align*}
Since $(\mb X_0, \mb Y_0)$ is the maximizer, for any $\mb{X} \in \mbb C^{d_1 \times r}$ and $\mb{Y} \in \mbb C^{d_2 \times r}$ satisfying $\mb X \mb Y^* = \mb M_0$, we can continue by
\begin{align*}
& (1-\eta) \, \norm{\mb X_0 - \tilde{\mb X}_0}_\F^2 
+ (1-\eta) \, \norm{\mb Y_0 - \tilde{\mb Y}_0}_\F^2 \\
& \leq \frac{1}{n} \, \norm{\mb A^* \mb X}_\F^2
- 2 \inp{\mb X, \tilde{\mb X}_0}
+ 2 \left\langle \mb X_0, \left(\mb I_{d_1} - \frac{1}{n} \mb A \mb A^*\right) \tilde{\mb X}_0 \right\rangle
+ \frac{1}{n} \, \norm{\mb A^* \tilde{\mb X}_0}_\F^2 \\
& \quad + \frac{1}{n} \, \norm{\mb B^* \mb Y}_\F^2
- 2 \inp{\mb Y, \tilde{\mb Y}_0}
+ 2 \left\langle \mb Y_0, \left(\mb I_{d_2} - \frac{1}{n} \mb B \mb B^*\right) \tilde{\mb Y}_0 \right\rangle
+ \frac{1}{n} \, \norm{\mb B^* \tilde{\mb Y}_0}_\F^2 \\
& = \frac{1}{n} \, \norm{\mb A^* (\mb X - \tilde{\mb X}_0)}_\F^2 
+ 2 \left\langle \mb X_0 - \mb X, \left(\mb I_{d_1} - \frac{1}{n} \mb A \mb A^*\right) \tilde{\mb X}_0 \right\rangle \\
& \quad + \frac{1}{n} \, \norm{\mb B^* (\mb Y - \tilde{\mb Y}_0)}_\F^2 
+ 2 \left\langle \mb Y_0 - \mb Y, \left(\mb I_{d_2} - \frac{1}{n} \mb B \mb B^*\right) \tilde{\mb Y}_0 \right\rangle \\
& \leq (1+\eta) \, \norm{\mb X - \tilde{\mb X}_0}_\F^2 
+ 2\eta \, \norm{\mb X - \tilde{\mb X}_0}_\F \norm{\tilde{\mb X}_0}_\F \\
& \quad + (1+\eta) \, \norm{\mb Y - \tilde{\mb Y}_0}_\F^2 
+ 2\eta \, \norm{\mb Y - \tilde{\mb Y}_0}_\F \norm{\tilde{\mb Y}_0}_\F \,,
\end{align*}
where the last step follows from \eqref{eq:isometry}. 
Finally note that
\begin{align*}
& \norm{\mb X - \tilde{\mb X}_0}_\F \norm{\tilde{\mb X}_0}_\F
+ \norm{\mb Y - \tilde{\mb Y}_0}_\F \norm{\tilde{\mb Y}_0}_\F\\
& \leq \left(\norm{\mb X - \tilde{\mb X}_0}_\F^2 + \norm{\mb Y - \tilde{\mb Y}_0}_\F^2 \right)^{1/2} 
\left( \norm{\tilde{\mb X}_0}_\F^2 + \norm{\tilde{\mb Y}_0}_\F^2 \right)^{1/2} \,.
\end{align*}
This completes the proof.

\section{Proofs of the supporting lemmas}
\subsection{Proof of Lemma~\ref{lem:sqrtDK}}
\label{sec:proof:lem:sqrtDK}
It suffices to show the upper bound for some $\mb{Q}$ in the \emph{orthogonal group} $\mc{O}_r(\mbb C)$. Let $\mb{Q}$ be given by
\[
\mb{Q} \in \argmin_{\mb{R} \in \mbb{R}^{r \times r}} \left\{ \left\| \begin{bmatrix} \mb{U}_0 \\ \mb{V}_0 \end{bmatrix} \mb{R} - \begin{bmatrix} \tilde{\mb{U}}_0 \\ \tilde{\mb{V}}_0 \end{bmatrix} \right\|_\F \st \mb{R}^* \mb{R} = \mb{I}_r \right\}.
\]
Then it follows that
\begin{align*}
& \left\| 
\begin{bmatrix} \tilde{\mb{U}}_0 \\ \tilde{\mb{V}}_0 \end{bmatrix} \tilde{\mb{\varSigma}}_0^{1/2}
- \begin{bmatrix} \mb{U}_0 \\ \mb{V}_0 \end{bmatrix} \mb{\varSigma}_0^{1/2} \mb{Q} \right\|_\F
= \left\| \begin{bmatrix} \tilde{\mb{U}}_0 \\ \tilde{\mb{V}}_0 \end{bmatrix} \tilde{\mb{\varSigma}}_0^{1/2} - \begin{bmatrix} \mb{U}_0 \\ \mb{V}_0 \end{bmatrix} \mb{Q} \mb{Q}^* \mb{\varSigma}_0^{1/2} \mb{Q} \right\|_\F \\
&\leq \left\| \left( \begin{bmatrix} \tilde{\mb{U}}_0 \\ \tilde{\mb{V}}_0 \end{bmatrix} - \begin{bmatrix} \mb{U}_0 \\ \mb{V}_0 \end{bmatrix} \mb{Q} \right) \tilde{\mb{\varSigma}}_0^{1/2} \right\|_\F
+ \left\| \begin{bmatrix} \mb{U}_0 \\ \mb{V}_0 \end{bmatrix} \mb{Q} (\tilde{\mb{\varSigma}}_0^{1/2} - \mb{Q}^* \mb{\varSigma}_0^{1/2} \mb{Q}) \right\|_\F \\
&\leq \sqrt{\sigma_1(\tilde{\mb{\varSigma}}_0)} \, \left\| \begin{bmatrix} \tilde{\mb{U}}_0 \\ \tilde{\mb{V}}_0 \end{bmatrix} - \begin{bmatrix} \mb{U}_0 \\ \mb{V}_0 \end{bmatrix} \mb{Q} \right\|_\F
+ \left\| \tilde{\mb{\varSigma}}_0^{1/2} - \mb{Q}^* \mb{\varSigma}_0^{1/2} \mb{Q} \right\|_\F.
\end{align*}

Since $\tilde{\mb{\varSigma}}_0^{1/2}$ (resp. $\mb{Q}^* \mb{\varSigma}_0^{1/2} \mb{Q}$) is the matrix square root of $\tilde{\mb{\varSigma}}_0$ (resp. $\mb{Q}^* \mb{\varSigma}_0 \mb{Q}$), by the perturbation bound by Schmitt \cite[equation (1.3)]{schmitt1992perturbation}, we obtain
\[
\left\| \tilde{\mb{\varSigma}}_0^{1/2} - \mb{Q}^* \mb{\varSigma}_0^{1/2} \mb{Q} \right\|_\F
\leq \frac{\left\| \tilde{\mb{\varSigma}}_0 - \mb{Q}^* \mb{\varSigma}_0 \mb{Q} \right\|_\F}{\sqrt{\sigma_r(\tilde{\mb{\varSigma}}_0)} + \sqrt{\sigma_r(\mb{\varSigma}_0)}}
\leq \frac{\left\| \tilde{\mb{\varSigma}}_0 - \mb{Q}^* \mb{\varSigma}_0 \mb{Q} \right\|_\F}{\sqrt{\sigma_r(\mb{\varSigma}_0)}}.
\]

On the other hand, by the triangle inequality, we have
\begin{align*}
\| \tilde{\mb{M}}_0 - \mb{M}_0 \|_\F
 & = \| \tilde{\mb{U}}_0 \tilde{\mb{\varSigma}}_0 \tilde{\mb{V}}_0^* - \mb{U}_0 \mb{Q} \mb{Q}^* \mb{\varSigma}_0 \mb{Q} \mb{Q}^* \mb{V}_0^* \|_\F \\
&\geq \| \mb{U}_0 \mb{Q} (\tilde{\mb{\varSigma}}_0 - \mb{Q}^* \mb{\varSigma}_0 \mb{Q}) \tilde{\mb{V}}_0^* \|_\F
- \| (\tilde{\mb{U}}_0 - \mb{U}_0 \mb{Q}) \tilde{\mb{\varSigma}}_0 \tilde{\mb{V}}_0^* \|_\F
\\
& - \| \mb{U}_0 \mb{Q} \mb{Q}^* \mb{\varSigma}_0 \mb{Q} (\tilde{\mb{V}}_0 - \mb{V}_0 \mb{Q})^* \|_\F \\
&\geq \| \tilde{\mb{\varSigma}}_0 - \mb{Q}^* \mb{\varSigma}_0 \mb{Q} \|_\F
- \sigma_1(\tilde{\mb{\varSigma}}_0) \, \| \tilde{\mb{U}}_0 - \mb{U}_0 \mb{Q} \|_\F
- \sigma_1(\mb{\varSigma}_0) \, \| \tilde{\mb{V}}_0 - \mb{V}_0 \mb{Q} \|_\F.
\end{align*}

By rearranging the above inequality, we obtain
\[
\| \tilde{\mb{\varSigma}}_0 - \mb{Q}^* \mb{\varSigma}_0 \mb{Q} \|_\F
\leq \| \tilde{\mb{M}}_0 - \mb{M}_0 \|_\F
+ \sigma_1(\tilde{\mb{\varSigma}}_0) \, \| \tilde{\mb{U}}_0 - \mb{U}_0 \mb{Q} \|_\F
+ \sigma_1(\mb{\varSigma}_0) \, \| \tilde{\mb{V}}_0 - \mb{V}_0 \mb{Q} \|_\F.
\]

By combining the results with Weyl's inequality, we obtain
\begin{align*}
& \left\| \begin{bmatrix} \mb{U}_0 \\ \mb{V}_0 \end{bmatrix} \mb{\varSigma}_0^{1/2} \mb{Q} - \begin{bmatrix} \tilde{\mb{U}}_0 \\ \tilde{\mb{V}}_0 \end{bmatrix} \tilde{\mb{\varSigma}}_0^{1/2} \right\|_\F
\leq \sqrt{\sigma_1(\tilde{\mb{\varSigma}}_0)} \, \left\| \begin{bmatrix} \tilde{\mb{U}}_0 \\ \tilde{\mb{V}}_0 \end{bmatrix} - \begin{bmatrix} \mb{U}_0 \\ \mb{V}_0 \end{bmatrix} \mb{Q} \right\|_\F \\
&\quad + \frac{\| \tilde{\mb{M}}_0 - \mb{M}_0 \|_\F
+ \sigma_1(\tilde{\mb{\varSigma}}_0) \, \| \tilde{\mb{U}}_0 - \mb{U}_0 \mb{Q} \|_\F
+ \sigma_1(\mb{\varSigma}_0) \, \| \tilde{\mb{V}}_0 - \mb{V}_0 \mb{Q} \|_\F}{\sqrt{\sigma_r(\mb{\varSigma}_0)}} \\
& \leq \frac{\| \tilde{\mb{M}}_0 - \mb{M}_0 \|_\F
}{\sqrt{\sigma_r(\mb{\varSigma}_0)}} 
+
\sqrt{2\sigma_1(\mb{\varSigma}_0)} \,
\left(1 + 2\sqrt{\frac{\sigma_1(\mb{\varSigma}_0)}{\sigma_r(\mb{\varSigma}_0)}} \right) \cdot \underbrace{\left\| \begin{bmatrix} \tilde{\mb{U}}_0 \\ \tilde{\mb{V}}_0 \end{bmatrix} - \begin{bmatrix} \mb{U}_0 \\ \mb{V}_0 \end{bmatrix} \mb{Q} \right\|_\F}_{\text{($\sharp$)}}.
\end{align*}

Finally, ($\sharp$) is bounded by a variant of the Davis-Kahan theorem by Dopico \cite[Theorem~2.1]{dopico2000note} as follows:
\begin{align*}
\left\|
\begin{bmatrix} \tilde{\mb{U}}_0 \\ \tilde{\mb{V}}_0 \end{bmatrix}
- \begin{bmatrix} \mb{U}_0 \\ \mb{V}_0 \end{bmatrix} \mb{Q}
\right\|_\F
\leq
\frac{\sqrt{2(\| (\tilde{\mb{M}}_0 - \mb{M}_0) \tilde{\mb{V}}_0 \|_\F^2 + \| \tilde{\mb{U}}_0^* (\tilde{\mb{M}}_0 - \mb{M}_0) \|_\F^2)}
}{\sigma_r(\mb{\varSigma}_0) - \|\tilde{\mb{M}}_0 - \mb{M}_0\|}.
\end{align*}

\subsection{Proof of Lemma~\ref{lem:upE}}
\label{sec:proof:lem:upE}
Since $(\mb{X}_0, \mb{Y}_0)$ is a maximizer to \eqref{eq:anchor_opt}, we have
\begin{equation}
\label{eq:lb1}
\begin{aligned}
& \langle \tilde{\mb{X}}_0, \mb{X}_0 - \mb{X} \rangle + \langle \tilde{\mb{Y}}, \mb{Y}_0 - \mb{Y} \rangle \\
& \geq \frac{1}{2n}
\left( \norm{\mb{A}^* \mb{X}_0}_\F^2 - \norm{\mb{A}^* \mb{X}}_\F^2 + \norm{\mb{B}^* \mb{Y}_0}_\F^2 - \norm{\mb{B}^* \mb{Y}}_\F^2
\right) \\
& \geq \frac{1}{n} \, \langle \mb{A} \mb{A}^* \mb{X}, \mb{X}_0\!-\!\mb{X} \rangle + \frac{1-\eta}{2} \, \norm{\mb{X}_0\!-\! \mb{X}}_\F^2 
+ \frac{1}{n} \, \langle \mb{B} \mb{B}^* \mb{Y}, \mb{Y}_0\!-\!\mb{Y} \rangle + \frac{1-\eta}{2} \, \norm{\mb{Y}_0\!-\!\mb{Y}}
_\F^2\,,
\end{aligned}
\end{equation}
where the second inequality follows from the strong convexity of $\|\mb{A}^* \mb{X}_0\|_\F^2$ and $\|\mb{B}^* \mb{Y}_0\|_\F^2$, which are quadratic functions of $\mb{X}_0$ and $\mb{Y}_0$, respectively. Then \eqref{eq:lb1} is rearranged to
\begin{align*}
& \frac{1-\eta}{2} \, \norm{\mb{X}_0 - \mb{X}}_\F^2
+ \frac{1-\eta}{2} \, \norm{\mb{Y}_0 - \mb{Y}}_\F^2 \\
& \leq
\langle \tilde{\mb{X}} - \frac{1}{n} \mb{A} \mb{A}^* \mb{X}, \mb{X}_0 - \mb{X} \rangle + \langle \tilde{\mb{Y}}_0 - \frac{1}{n} \mb{B} \mb{B}^* \mb{Y}, \mb{Y}_0- \mb{Y} \rangle \\
& \leq \sqrt{\norm{\mb{X}_0 - \mb{X}}_\F^2
+ \norm{\mb{Y}_0-\mb{Y}}_\F^2} \cdot
\sqrt{\norm{\tilde{\mb{X}}_0 - \frac{1}{n} \mb{A} \mb{A}^* \mb{X}}_\F^2 + \norm{\tilde{\mb{Y}}_0 - \frac{1}{n} \mb{B} \mb{B}^* \mb{Y}}_\F^2}\,,
\end{align*}
where the last step follows from the Cauchy-Schwarz inequality.

Therefore it follows that
\begin{equation*}
\norm{\left[\tilde{\mb{X}}_0 - \frac{1}{n}\mb{A}\mb{A}^* \mb{X} ;\ \tilde{\mb{Y}}_0 - \frac{1}{n}\mb{B}\mb{B}^* \mb{Y}\right]}_\F
\ge \frac{1-\eta}{2} \, 
\norm{\left[\mb{X}_0 - \mb{X} ;\ \mb{Y}_0 - \mb{Y}\right]}_\F \,.
\end{equation*}

Finally by the triangle inequality we obtain
\begin{align*}
& \norm{\left[\tilde{\mb{X}}_0 - \frac{1}{n}\mb{A}\mb{A}^* \mb{X}_0 ;\ \tilde{\mb{Y}}_0 - \frac{1}{n}\mb{B}\mb{B}^* \mb{Y}_0\right]}_\F \\
& \leq 
\norm{\left[\tilde{\mb{X}}_0 - \frac{1}{n}\mb{A}\mb{A}^* \mb{X} ;\ \tilde{\mb{Y}}_0 - \frac{1}{n}\mb{B}\mb{B}^* \mb{Y} \right]}_\F
+ \norm{\left[\frac{1}{n}\mb{A}\mb{A}^* (\mb{X}_0 - \mb{X}) ;\ \frac{1}{n}\mb{B}\mb{B}^* (\mb{Y}_0 - \mb{Y})\right]}_\F \\
& \leq 
\norm{\left[\tilde{\mb{X}}_0 - \frac{1}{n}\mb{A}\mb{A}^* \mb{X} ;\ \tilde{\mb{Y}}_0 - \frac{1}{n}\mb{B}\mb{B}^* \mb{Y} \right]}_\F
+ (1+\eta) \, 
\norm{\left[\mb{X}_0 - \mb{X} ;\ \mb{Y}_0 - \mb{Y}\right]}_\F \\
& \leq 
\norm{\left[ \tilde{\mb{X}}\!-\!\frac{1}{n}\mb{A}\mb{A}^* \mb{X} ;\ \tilde{\mb{Y}}_0\!-\!\frac{1}{n}\mb{B}\mb{B}^* \mb{Y}\right]}_\F
+ \frac{2(1\!+\!\eta)}{1\!-\!\eta} \, 
\norm{\left[ \tilde{\mb{X}}_0\!-\! \frac{1}{n}\mb{A}\mb{A}^* \mb{X} ;\ \tilde{\mb{Y}}\!-\!\frac{1}{n}\mb{B}\mb{B}^* \mb{Y} \right]}_\F \\
& \leq \left(1+\frac{2(1+\eta)}{1-\eta}\right)
\left\{
\norm{\left[\tilde{\mb{X}}_0 - \mb{X} ;\ \tilde{\mb{Y}}_0 - \mb{Y}\right]}_\F
+ \norm{\left[ \mb{X} - \frac{1}{n}\mb{A}\mb{A}^* \mb{X} ;\ \mb{Y} - \frac{1}{n}\mb{B}\mb{B}^* \mb{Y} \right]}_\F
\right\} \,.
\end{align*}
Finally note that
\[
\norm{\left[ \mb{X} - \frac{1}{n}\mb{A}\mb{A}^* \mb{X} ;\ \mb{Y} - \frac{1}{n}\mb{B}\mb{B}^* \mb{Y} \right]}_\F
\leq \eta\,\norm{[\mb{X} ;\ \mb{Y}]}_\F \,,
\]
which completes the proof.

\bibliographystyle{abbrvnat}
\bibliography{refs}

\end{document}